\documentclass[12pt,reqno,a4paper]{amsart}
\usepackage{
    amsmath,  amsfonts, amssymb,  amsthm,   amscd,
    gensymb,  graphicx, comment,  etoolbox, url,
    booktabs, stackrel, mathtools,enumitem, mathdots,  microtype, lmodern,    mathrsfs, graphicx, tikz,  longtable,tabularx, float, tikz, pst-node, tikz-cd, multirow, tabularx, amscd,  bm, array, makecell, diagbox, booktabs,ragged2e, caption, subcaption }
\usepackage{makecell,slashbox}
\usepackage{xcolor}

\usepackage[all]{xy}
\usepackage{graphicx}
\usepackage[utf8]{inputenc}
\usepackage{microtype, fullpage, wrapfig,textcomp,mathrsfs,csquotes,fbb}
\usepackage[colorlinks=true, linkcolor=blue, citecolor=blue, urlcolor=blue, breaklinks=true]{hyperref}
\usepackage[capitalise]{cleveref}
\usepackage{todonotes}
\usetikzlibrary{positioning}
\usetikzlibrary{shapes,arrows.meta,calc}
\usetikzlibrary{arrows}

\newtheorem{theorem}{Theorem}[section]

\newtheorem{corollary}[theorem] {Corollary}
\newtheorem{definition}[theorem]{Definition}
\newtheorem{example}[theorem]{Example}
\newtheorem{lemma}[theorem]{Lemma}

\newtheorem{proposition}[theorem]{Proposition}
\newtheorem{remark}[theorem]{Remark}

\newcommand\Q{\mathbb{Q}}
\newcommand\pa{\partial}
\newcommand\B{\mathcal{B}}
\newcommand\R{\mathbb{R}}
\newcommand\Z{\mathbb{Z}}
\newcommand\C{\mathbb{C}}
\newcommand{\dTC}{\mathrm{dTC}}

\newcommand{\supp}{\mathrm{supp}}
\newcommand{\dcat}{\mathrm{dcat}}
\newcommand{\dsct}{\mathrm{dsecat}}
\newcommand{\dsecat}{\mathrm{dsecat}}
\newcommand{\TC}{\mathrm{TC}}
\newcommand{\cat}{\mathrm{cat}}

\newcommand{\secat}{\mathrm{secat}}
\newcolumntype{x}[1]{>{\centering\arraybackslash}p{#1}}

\begin{document}
\title[On the complexity of parametrized motion planning algorithms]{On the complexity of
\\
parametrized motion planning algorithms}

\author[N. Daundkar]{Navnath Daundkar}
\address{Navnath Daundkar, Department of Mathematics, Indian Institute of Technology Madras, Chennai, India.}
\email{navnath@iitm.ac.in}
\author[E. Jauhari]{Ekansh Jauhari}
\address{Ekansh Jauhari, Department of Mathematics, University of Florida, 358 Little Hall, Gainesville, FL 32611, USA.}
\email{ekanshjauhari@ufl.edu}

\thanks{}

\begin{abstract}
We study a probabilistic variant of the $r$-th sequential parametrized topological complexity, which bounds this classical invariant from below and measures the difficulty in constructing permissive parametrized motion planning algorithms. On one hand, we use cohomology to show that this new invariant behaves similarly to the classical invariant on Fadell--Neuwirth fibrations and oriented sphere bundles; on the other hand, we use equivariant homotopy theory to prove that its behavior is wildly different on bundles whose fibers are real projective spaces and whose structure groups are special orthogonal groups. We also explore several other features of our invariant and its relationships with various other invariants motivated by topological robotics.
\end{abstract}

%
%
%
%



%
%
%
%

\maketitle

\section{Introduction}\label{sec:intro}
In robotics, an autonomously functioning mechanical system, such as a robot, is typically required to move from a specified initial position to a desired target position within its configuration space\hspace{0.3mm}\footnote{\hspace{0.3mm}The configuration space of a mechanical system is the space of all possible positions or configurations of the system, endowed with a suitable topology.} in an optimal manner. By \emph{optimal}, we mean executing the motion with the fewest possible instabilities or discontinuities, which are governed by the topology of the configuration space. This is the classical motion planning problem, where, given an input of the initial and the final positions of the system, a motion of the system from the initial to the final position is expected as the output. A solution to this problem is given by a \emph{motion planning algorithm}, which is a rule that enables a system to autonomously navigate between any two given positions inside its configuration space. Farber's \emph{topological complexity}~\cite{F} is a homotopy invariant that provides a numerical measure of the difficulty in constructing a motion planning algorithm on a given mechanical system.

The motion planning problem becomes more complicated when an algorithm is to be constructed to navigate a system between multiple (or more than two) given positions in a prescribed order at specific instances of time. For such \emph{sequential} motions, Rudyak's \emph{sequential topological complexity}~\cite{Rudyak} extends Farber's notion in a natural way by giving a similar numerical measure of the complexity of constructing the required sequential motion planning algorithms.  

For certain advanced systems that are capable of breaking and reassembling repeatedly while in motion\hspace{0.3mm}\footnote{\hspace{0.3mm}Examples of such systems may include robots like Terminator 2 from the eponymous movie, and a flock of identical drones that is viewed as a single mechanical system.}, a new approach to the motion planning problem was suggested recently by Dranishnikov and Jauhari in~\cite{D-J},~\cite{Jau1}, and independently and simultaneosuly, by Knudsen and Weinberger in~\cite{K-W}. This approach allows a mechanical system to break into finitely many smaller, \emph{unordered} weighted pieces while navigating between a given input of its positions, provided the system reassembles while attaining each required position in the prescribed order at specific instances of time. A homotopy invariant, \emph{sequential distributional complexity}, measures the difficulty in constructing such a motion planning algorithm on a given system.

In this breaking and reassembling type of motion, each piece of the system represents a \emph{state of the motion} of the system, and the weight of a piece represents the probability of the system being in that corresponding state of motion. Then this motion corresponds to an unordered probability distribution. So, sequential distributional complexity can also be seen as a measure of the minimal level of \emph{randomness} in such a motion planning algorithm (see also~\cite{K-W}\hspace{0.3mm}\footnote{\hspace{0.3mm}Strictly speaking, in~\cite{K-W}, such an interpretation is given for a related invariant, namely the \emph{sequential analog complexity}. However, it is now known that this invariant coincides with the sequential distributional complexity on compact metric spaces, among other spaces,~\cite{K-W2},~\cite[Page 2]{Dr}.}). We invite the reader to compare this probabilistic interpretation with the one given by Farber in~\cite{F2} for conventional motion, where it is explained that conventional motion of a system corresponds, instead, to an \emph{ordered} probability distribution.

As is intuitively clear, this breaking and reassembling type of motion (which we may sometimes refer to as a \emph{distributed} motion) could be easier to plan than conventional motion, where the entire system must navigate as a single unit. Thus, the sequential distributional complexity of a space cannot exceed its sequential topological complexity. For instance, for the simple problem of planning the motion of a rotating line in the $(n+1)$-dimensional Euclidean space $\R^{n+1}$ that is fixed along a revolving joint at a base point, while the topological complexity is equal to the immersion dimension of the real projective space $\R P^n$ when $n\ne1,3,7$ and is equal to $n$ otherwise due to~\cite{FTY}; the distributional complexity is always $1$, i.e., the least possible, as observed in both~\cite{D-J} and~\cite{K-W}. Therefore, motions planned with this approach could, at times, be much more efficient.

\subsection{Parametrized motion planning}\label{intro: parametrized}
Typically, mechanical systems, such as robots, are subject to variable external conditions while in motion. An example of this situation is when a robot is expected to navigate in a physical space by avoiding (possibly non-stationary) objects with unknown a priori positions\hspace{0.3mm}\footnote{\hspace{0.3mm}See~\cite[Page 2]{C-F-W} for an explicit such practical situation.}. Such external conditions are viewed as parameters and included in the input of a motion planning algorithm, which is then called a \emph{parametrized motion planning algorithm}. Of course, such algorithms are more flexible and universal, in the sense that they control the navigation of the system in a variety of situations. To numerically measure the difficulty in constructing parametrized motion planning algorithms, the notion of \emph{parametrized topological complexity} was introduced and studied by Cohen, Farber, and Weinberger in~\cite{C-F-W}. Subsequently, the notion of \emph{sequential parametrized topological complexity} was developed by Farber and Paul in~\cite{F-P1} as a natural extension to measure the complexity of parametrized algorithms that control the navigation of systems moving between multiple positions while subject to constraints.

In this framework, the constraints acting on the system are encoded by an auxiliary topological space $B$. The associated \emph{universal configuration space} is modeled by a fibration $p\colon E \to B$, where $E$ is interpreted as a union of configuration spaces indexed by the parameters in $B$, i.e., $E = \bigcup_{b \in B} F_b$. Here, each fiber $F_b$ represents the space of all achievable configurations of the system subject to the constraint $b \in B$. Then, a parametrized motion planning algorithm on a system takes as input a finite sequence of positions of the system lying in the same fiber (this means under the same constraint), and returns as output a continuous path (or motion of the system) between them that also lies entirely within that fiber, thereby respecting the imposed constraint throughout. 

In this paper, we propose a ``distributional'' approach to parametrized motion planning to reduce the complexity of parametrized motion planning algorithms. More precisely, we want to study parametrized motion planning algorithms on advanced systems that are capable of breaking and reassembling while moving. These algorithms will be permissive in the sense that they will allow our system to break into smaller, unordered pieces so that it can execute a sequential motion between input positions subject to the same constraint. This kind of breaking and reassembling of a system will produce a motion whose complexity will potentially be less than that of conventional parametrized motion, since the motion, as explained above, will correspond to an unordered probability distribution of paths whose images lie in a common fiber (or external constraint). 

To measure the difficulty (or, in probabilistic terms, to record the minimal degree of the unordered randomness) in constructing such more flexible parametrized motion planning algorithms on advanced systems, we introduce an invariant that we call, for obvious reasons, \emph{sequential parametrized distributional complexity}. For a fibration, which is viewed as a model of the universal configuration space of a system that is subject to external parameters, this invariant gives the minimal number of unordered weighted pieces into which the system must decompose so that the pieces can travel, continuously and independently, between any given fixed-length sequence of positions in a prescribed order while respecting the corresponding external conditions. Therefore, in a sense, sequential parametrized distributional complexity quantifies the \emph{instabilities} of parametrized distributed motion planning algorithms that work simultaneously across all fibers.

\subsection{Features of sequential parametrized distributional complexity}\label{subsec: features}
Let us now list some mathematical features of our new invariant that we find in this paper and compare it with other, previously studied invariants in topological robotics.

For a given fibration $p\colon E\to B$, just like the $r$-th sequential parametrized topological complexity, which is denoted by $\TC_r[p\colon E\to B]$, our new invariant, the $r$-th sequential parametrized distributional complexity, denoted $\dTC_r[p\colon E\to B]$, is defined in terms of the fibration $\Pi_r^E\colon E^I_B\to E^r_B$, which we describe in~\Cref{subsect: parametrized TC}. In detail, $\dTC_r[p\colon E\to B]$ is the \emph{distributional sectional category} of $\Pi_r^E$. The latter is a homotopy invariant of fibrations that we study in~\Cref{sect: dsecat} and a distributional analogue of (and hence, a lower bound to) the classical \emph{sectional category} or \emph{Schwarz genus} recalled in~\Cref{subsec: secat}. 

Basically, $\dTC_r[p\colon E\to B]$ is defined using finitely-supported probability measures in the space $E^I_B$ which lie in a single fiber of the fibration $\Pi_r^E\colon E^I_B\to E^r_B$, see~\Cref{sec: sequential parametrized distributional complexity}. This interpretation allows us to prove several basic homotopy-theoretic properties of the $r$-th sequential parametrized distributional complexity in~\Cref{subsec: basic properties}, such as its homotopy invariance and its connection with the $r$-th sequential (non-parametrized) distributional complexity of the fiber $F$ of a fibration $p\colon E\to B$, which is denoted by $\dTC_r(F)$. When $p\colon E\to B$ is a principal $G$-bundle (so that the fiber $F$ is the path-connected topological group $G$), we obtain, analogous to the classical setting, an equality 
\[
\dTC_r[p\colon E\to B]=\dTC_r(G),
\] 
see~\Cref{prop: seq dtc of principal G bundles}. Using this, we show in~\Cref{ex: temporary} that for all (infinitely many) principal $SO(3)$-bundles $p\colon E\to B$, there is the strict inequality 
\begin{equation}\label{eq: different}
\dTC_r[p\colon E\to B]<\TC_r[p\colon E\to B]
\end{equation}
for each $r\ge 2$. This establishes that the notions of sequential parametrized topological and distributional complexities are mutually different, so the theory of the latter is of independent interest. In terms of motion planning, it means that for any mechanical system whose space of configurations subject to external constraints is $\R P^3\cong SO(3)$, and whose universal configuration space is modeled by a principal $SO(3)$-bundle, planning conventional $r$-th sequential parametrized motion requires exactly $3r-3$ \emph{rules}, while planning the breaking and reassembling type of the $r$-th sequential parametrized motion requires at most $2r+1$ rules, so that the new motion is less complicated. 

Inspired by the case of (sequential) parametrized topological complexity~\cite{C-F-W},~\cite{F-P1}, we then find \emph{sharp} cohomological lower bounds to $\dTC_r[p\colon E\to B]$ for each $r$ and fibration $p\colon E\to B$ in~\Cref{sec: lower bounds} using the cohomology ring (with arbitrary coefficients) of the symmetric products $SP^k(E^r_B)$ of the space of parametrized inputs $E^r_B$. In the case when $E$ and $B$ are finite CW complexes and the coefficient field is $\Q$, we show in~\Cref{prop: zcl bound} that our lower bound to $\dTC_r[p\colon E\to B]$ coincides with the well-known lower bound to $\TC_r[p\colon E\to B]$, which is simply in terms of the rational cohomology of $E$ and $E^r_B$. These results utilize methods developed in~\cite{D-J},~\cite{Jau1}, and~\cite{Jau2}.

In~\Cref{subsec: fadell-neuwirth}, we use our cohomological lower bounds to compute the $r$-th sequential parametrized distributional complexity of the \emph{Fadell--Neuwirth fibrations} 
\[
p\colon F(\R^d,m+n)\to F(\R^d,m), \hspace{5mm} p(x_1,\ldots,x_m,x_{m+1},\ldots,x_{m+n})=(x_1,\ldots,x_m)
\]
for each $r,m,d\ge 2$ and $n\ge 1$. It turns out that for these fibrations, there is the equality 
\[
\dTC_r[p\colon F(\R^d,m+n)\to F(\R^d,m)]=\TC_r[p\colon F(\R^d,m+n)\to F(\R^d,m)].
\]
Our computation uses the hard work done in~\cite{C-F-W},~\cite{PTCcolfree},~\cite{F-P1},~\cite{F-P2} for the $\TC_r$ values. Here, $F(\R^d,k)$ is the configuration space of $k$ distinct ordered points inside the Euclidean space $\R^d$, and $p$ is the map that forgets the last $n$ configurations. While these fibrations are important in topology~\cite{F-N}, they are also relevant in robotics. Consider the problem of planning a collision-free motion of $n$ robots in $\R^d$ in the presence of $m$ obstacles. A solution to this problem is given by a section of the Fadell--Neuwirth fibration $p\colon F(\R^d,m+n)\to F(\R^d,m)$. Hence, the difficulty in planning such an $r$-sequential motion in the conventional (resp. the new) way is equal to the $r$-th sequential parametrized topological (resp. distributional) complexity of $p$. It turns out that in any case, this parametrized motion requires $rn+m-2$ rules if $d$ is even (for example, if the collision-free motion is expected in the $2$-dimensional plane) and $rn+m-1$ rules if $d$ is odd (for example, if the above motion is expected in the $3$-dimensional space), see~\Cref{prop: f-n fibration}. In particular, the new approach to parametrized motion planning via the breaking and reassembling type of motion \emph{does not} reduce the complexity of this problem. 

We show in~\Cref{subsec: sphere bundle} that a similar phenomenon as above is observed when the universal configuration space of a system is modeled by an oriented sphere bundle $\dot\xi\colon\dot E\to B$ (coming from a vector bundle) over a sufficiently low-dimensional, ``nice'' base $B$, that admits a section. This is done by studying, in the spirit of~\cite{F-W} and~\cite{F-P-sphere-bundle}, the sequential parametrized distributional complexities of these sphere bundles. Our main result of this subsection (\Cref{prop: sphere bundle}) gives the cohomological lower bound 
\[
\dTC_r[\dot\xi\colon\dot E\to B]\ge \mathfrak{h}(\mathfrak{e}(\ddot\xi))+r-1,
\]
where $\ddot\xi\colon\ddot E\to \dot E$ is the sphere bundle corresponding to $\dot\xi$, $\mathfrak{e}(\ddot\xi)$ is the \emph{Euler class} of $\ddot\xi$, and $\mathfrak{h}(\mathfrak{e}(\ddot\xi))$ is the \emph{height} of this class. The above bound agrees with the corresponding lower bound to $\TC_r[\dot\xi\colon\dot E\to B]$. This computation is performed using the Leray--Hirsch description~\cite{Sp},~\cite{Hatcher} of the cohomology ring of the space $\dot E^r_B$ by proceeding along the lines of~\cite{F-P-sphere-bundle},~\cite{F-W}. Our examples in~\Cref{sec: ex and comput} illustrate that, just like their classical counterparts, sequential parametrized distributional complexities of fiber bundles can be as large as desired, and that the complexity of planning a sequential distributed motion on a system subject to external parameters can be arbitrarily more challenging than planning its non-parametrized sequential distributed motion, see~\Cref{rem: different from fiber} and~\Cref{rem: arbitrarily large}.

\Cref{ex: temporary}, which gives principal bundles satisfying~\eqref{eq: different}, uses the difference in the $r$-th sequential topological and distributional complexities of $\R P^3$ (and generally, of all real projective spaces) found in~\cite{D-J},~\cite{K-W},~\cite{Jau1}. To produce ``genuinely new'' examples of fibrations satisfying~\eqref{eq: different}, we introduce as a tool in~\Cref{sec: equi dTC} the notion of \emph{sequential equivariant distributional complexity} for spaces endowed with continuous group actions (this invariant also has interpretations in terms of motion planning for advanced robotics systems whose configuration spaces have symmetries, see~\Cref{rmk: equi TC motion planning} for an analogy). 

For this, we first study the distributional analogue of the \emph{equivariant sectional category} of Colman and Grant~\cite{EqTC}, which is revisited in~\Cref{subsec: equi sect cat}. Then the $r$-th sequential equivariant distributional complexity of a $G$-space $X$, denoted $\dTC_{G,r}(X)$, is defined in~\Cref{sec: equi dTC} as the \emph{equivariant distributional sectional category} of the $G$-fibration $\pi_r^X\colon X^{[0,1]} \to X^r$ that evaluates each path at $t_i=\tfrac{i-1}{r-1}$ for $1\le i \le r$. In~\Cref{subsec: eq dtc of products of spheres}, we study the invariants $\dTC_{\Z_2,r}$ and $\TC_{\Z_2,r}$ for some $\Z_2$-actions on products of spheres, where the latter is the \emph{sequential equivariant topological complexity} defined by Bayeh and Sarkar in~\cite{B-S} (see also~\cite{EqTC}), and observe that they coincide in various situations. However, it turns out that the equivariant distributional invariants are different from their classical counterparts in general. Indeed, we define an action of the special orthogonal group $SO(n)$ on $\R P^n$ for each $n$ and obtain in~\Cref{prop: SOn eq TC of RPn} the startling computation  
\[
\dTC_{SO(n)}(\R P^n)=1 < \TC_{SO(n)}(\R P^n) \ \text{ for all } \ n\ge 2.
\] 
Next, we explain that, almost as in the classical setting~\cite{F-O}, the sequential parametrized distributional complexities of some fibrations are intimately related to the sequential equivariant distributional complexities of their respective fibers, hence justifying the study of the latter notion in this paper. More precisely, for any $G$-space $F$ and principal $G$-bundle $q\colon P\to B$, we look at the $F$-\emph{associate bundle} of $q$, denoted $p\colon F\times_G \hspace{0.4mm}P\to P/G$, whose fiber is $F$ and structure group is $G$, and in analogy with a result of Farber and Oprea from~\cite{F-O}, we establish in~\Cref{thm: distributional analogue of Farber-Oprea} that
\[
\dTC_r[p\colon F\times_G \hspace{0.3mm}P\to P/G]\le\left(\dTC_{G,r}(F)+1\right)^2-1.
\]
This relationship, along with the results of~\Cref{subsec: equi dtc rpn}, finally gives us the ``genuinely new'' examples satisfying~\eqref{eq: different} as we take $G=SO(n)$ and $F=\R P^n$. This is explained in~\Cref{sec:new upper bound}, where we end this paper by proving that there are infinitely many non-principal bundles whose sequential parametrized complexity values in the distributional and the classical cases are different; in fact, arbitrarily far apart.

\subsection{Notations and conventions} All topological spaces considered in this paper are path-connected, and all rings are commutative with unity. We use the term ``maps'' for continuous functions and homomorphisms, ``sections'' for continuous right inverses in the category of topological spaces, and ``fibrations'' for Hurewicz fibrations. We denote the cup product of cohomology classes $x$ and $y$ by $xy$.

\section{Preliminaries and background material}\label{sec: prelims}
\subsection{Sectional category of fibrations}\label{subsec: secat}
The notion of sectional category was introduced by Schwarz in~\cite{Sch} as ``genus'', and then developed further by Berstein and Ganea in~\cite{B-G}, and subsequently by James in~\cite{James}. We recall that the \textit{sectional category} of a fibration $p\colon E \to B$, denoted $\secat(p)$, is the smallest integer $n$ such that $B$ can be covered by $n+1$ open sets $W_i$ over each of which there exists a partial section $s_i\colon W_i\to E$ of $p$.

When dealing with paracompact spaces, one can reformulate the definition of $\secat(p)$ as follows. Define
\[
\ast^k_B \hspace{0.5 mm}E = \left\{\sum_{i=1}^{k} \lambda_{i} e_{i} \hspace{1mm} \middle| \hspace{1mm} e_{i} \in E, \hspace{0.5mm} \sum_{i=1}^{k}\lambda_{i} = 1, \hspace{0.5mm} \lambda_{i} \geq 0, \hspace{0.5mm} p\left(e_{i}\right)= p\left(e_{j}\right)\right\}
\]
to be the iterated fiberwise join of $k$-copies of $E$ along $p$, and similarly, define the iterated fiberwise join of $k$-copies of $p$, denoted $\ast^k_B \hspace{0.5mm}p \colon \ast^k_B \hspace{0.5 mm}E \to B$, such that
\[
\ast^k_B \hspace{0.5mm}p \left(\sum_{i=1}^{k} \lambda_{i} e_{i}\right) = p\left(e_i \right)
\]
for any $i$ with $\lambda_{i} > 0$. Then we have the following well-known characterization of $\secat(p)$ due to Schwarz,~\cite{Sch}.
\begin{proposition}[\protect{\cite[Proposition 8.1]{James}}]\label{prop: sectional category}
     For a fibration $p\colon E\to B$, where $B$ is paracompact, $\secat(p)\le n$ if and only if the fibration $\ast^{n+1}_B \hspace{0.5mm}p \colon \ast^{n+1}_B \hspace{0.5 mm}E \to B$ admits a section.
\end{proposition}

We record some well-known properties of sectional category.
\begin{enumerate}
    \item If two fibrations $p\colon E\to B$ and $p'\colon E'\to B'$ are fiberwise homotopy equivalent, then $\secat(p)=\secat(p')$.\label{(1)}
    \item For any fibration $p\colon E\to B$ and map $f\colon A\to B$, $\secat(f^*p)\le\secat(p)$, where $f^*p$ denotes the pullback of $p$ with respect to $f$.\label{(2)}
    \item If a fibration $p\colon E\to B$ factors through another fibration $p'\colon E'\to B$, then $\secat(p')\le \secat(p)$.\label{(3)}
\end{enumerate}

Consider a pointed space $(X,x_0)$. Let $I=[0,1]$  and $X^I$ be the path space of $X$ endowed with the compact-open topology. Let $P_0(X)=\{\phi\in X^I\mid\phi(1)=x_0\}$ be the based path space. Define the evaluation fibration $p^X\colon P_0(X)\to X$ such that $p^X(\phi)=\phi(0)$. Then the \emph{Lusternik--Schnirelmann category} of $X$, denoted $\cat(X)$, is defined as $\cat(X):=\secat(p^X)$, see, for example,~\cite{James}. We refer to~\cite{CLOT} for a comprehensive survey on Lusternik--Schnirelmann category and related homotopy invariants.

For each integer $r\ge 2$, let us define the path fibration $\pi_r^X\colon X^I\to X^r$ such that
\begin{equation}\label{eq: sequential free path fibration}
    \pi_r^X(\phi)=\left(\phi(0),\phi\left(\frac{1}{r-1}\right),\cdots,\phi\left(\frac{r-2}{r-1}\right),\phi(1)\right).
\end{equation}
Then for any $r\ge 2$, the $r$\emph{-th sequential topological complexity} of $X$, denoted $\TC_r(X)$, is defined as $\TC_r(X):=\secat(\pi_r^X)$, see~\cite{F},~\cite{Rudyak}. Furthermore, a section of $\pi_r^X$ is called a \emph{sequential motion planning algorithm}. As described in the introduction, $\TC_r$ admits interpretations in terms of motion planning on robotic systems.

Since its inception, various versions and variants of (sequential) topological complexity have been studied, many of which were partly inspired by some advanced, more realistic motion planning problems. We now discuss one such variant that stems from the motion planning problem where external parameters are involved.

\subsection{Sequential parametrized topological complexity}\label{subsect: parametrized TC}

The notion of \emph{parametrized topological complexity} was introduced by Cohen, Farber, and Weinberger in~\cite{C-F-W} to study the motion planning problem in the presence of external variable conditions on the system (see~\Cref{intro: parametrized}). Recently, as a natural extension, the notion of \emph{sequential parametrized topological complexity} was introduced and studied by Farber and Paul in~\cite{F-P1}. We now briefly recall this notion. 

For a fibration $p\colon E\to B$ with fiber $F$, consider the space
\[E^I_B:=\{\gamma\in E^I \mid p\circ \gamma(t)=b ~\text{for some}~ b\in B ~\text{and for all}~ t\in[0,1] \}.\]
The fiber product corresponding to $p:E\to B$ is defined by \[E^r_B:=\{(e_1,\dots,e_r)\in E^r \mid p(e_i)=p(e_j) ~\text{ for }~ 1\leq i,j\leq r\}.\]

\begin{remark}\label{rmk: cw structure on erb}
\rm{When $p\colon E\to B$ is a fibration of CW complexes, then for any $r\ge 2$, the space $E^r_B$ is obtained as follows. First, $E^2_B$ is obtained as the pullback space of $p$ along $p$. By~\cite[Lemma~36]{Mather}, $E^2_B$ has the homotopy type of a CW complex. Then, $E^3_B$ is obtained from the pullback of $p$ along the fibration $p^*p\colon E^2_B\to E$, so that $E^3_B$ also has the homotopy of a CW complex for the same reason. Similarly, $E^r_B$ is obtained using $E^{r-1}_B$ for each $r\ge 4$, and it has the homotopy type of a CW complex. We also note that if $E$ and $B$ are finite CW complexes, then so is $E^r$ for each $r$, and hence, $E^r_B\subset E^r$ has the homotopy type of a finite CW complex. From now onwards, when $E$ and $B$ are (finite) CW complexes, we shall abuse terminology and simply say that $E^r_B$ \emph{is} a (finite) CW complex.}
\end{remark}

Similar to the map $\pi_r^X$ defined in~\eqref{eq: sequential free path fibration}, we define a map $\Pi_r^E\colon E^I_B\to E^r_B$ such that 
\begin{equation}\label{eq: parametrized endpoint map}
\Pi_r^E(\gamma)= \left(\gamma(0), \gamma\left(\frac{1}{r-1}\right),\dots,\gamma\left(\frac{r-2}{r-1}\right),\gamma(1)\right).  
\end{equation}
It follows from~\cite[Appendix]{PTCcolfree} that $\Pi_r^E$ is a fibration with fiber $\Omega F$. Here, $E^r_B$ can be viewed as the space of parametrized inputs of the system, and $E^I_B$ as the space of paths along which the system can traverse while respecting the external parameters. Then a section of $\Pi_r^E$ becomes a \emph{sequential parametrized motion planning algorithm}. 

\begin{definition}
\rm{For any $r\ge 2$, the $r$\emph{-th sequential parametrized topological complexity} of $p$, denoted $\TC_r[p\colon E\to B]$, is defined as $\TC_r[p\colon E\to B]:=\secat(\Pi_r^E)$.}
\end{definition}

In~\cite{F-P1},~\cite{F-P2}, and~\cite{F-P-sphere-bundle}, a theory for the sequential parametrized topological complexity was established, generalizing several results and computations from~\cite{C-F-W},~\cite{PTCcolfree},~\cite{F-W}, and~\cite{F-W2} on the original notion of parametrized topological complexity. We also refer to~\cite{GC} for an extension of this notion to fiberwise spaces.

\subsection{Sequential equivariant topological complexity}\label{subsec: equi sect cat}
The equivariant analogue of sectional category (\emph{cf.}~\Cref{subsec: secat}) was introduced by Colman and Grant in~\cite{EqTC}.

\begin{definition}[{\cite[Definition 4.1]{EqTC}}\label{def:eqsecat}]
\rm{Suppose $E$ and $B$ are $G$-spaces. Let $p \colon E\to B$ be a $G$-map. 
The \emph{equivariant sectional category} of $p$, denoted $\secat_G(p)$, is the smallest integer $k$ for which there exists a $G$-invariant open cover $\{U_1,\ldots, U_{k+1}\}$ of $B$ and $G$-maps $s_i \colon U_i \to E$ such that $p \circ s_i \simeq_G \iota_{U_i}$ for each $i$, where $\iota_{U_i} \colon U_i \hookrightarrow B$ is the inclusion.}
\end{definition}
In the above definition, the notation $\simeq_G$ denotes that $p\circ s_i$ and $i_{U_i}$ are homotopic via a homotopy that itself is a $G$-map.

Several basic properties of $\secat_G(p)$ have been studied by Arora and Daundkar in~\cite{A-D}. For example, the equivariant cohomological lower bound, the equivariant homotopy dimension-connectivity upper bound, the product inequality, etc., have been established.

The notion of sequential equivariant topological complexity was introduced by Bayeh and Sarkar in~\cite{B-S} as an extension of the notion of \emph{equivariant topological complexity} from~\cite{EqTC}. We now briefly recall its definition. 

Let $X$ be a $G$-space. Then the free path space $X^{I}$ is a $G$-space with the following action:
$$ 
    (g \alpha)(t) =  g\cdot \alpha(t) \ \text{ for } \ g\in G \ \text{ and } \ \alpha\in X^I.
$$  
Also, $X^r$ is a $G$-space with respect to the diagonal action. The fibration 
$\pi_r^X \colon X^I \to X^r$ described in~\eqref{eq: sequential free path fibration} is a $G$-fibration. This happens because $\Pi_r^X\colon X^I_B\to X^r_B$ is a $G$-fibration whenever $p\colon X\to B$ is a $G$-fibration (see~\cite[Proposition 4.2]{D-EqPTC}), so taking $B = \{\ast\}$ shows that $\pi^X_r$ is also a $G$-fibration. Then, the $r$\emph{-th sequential equivariant topological complexity} of a $G$-space $X$, denoted $\TC_{G,r}(X)$, is defined as $\TC_{G,r}(X):=\secat_G(\pi_r^X)$. It is easy to see that $\TC_r(X)\le\TC_{G,r}(X)$ for each $r\ge 2$. In general, this inequality can be strict!

\begin{remark}\label{rmk: equi TC motion planning}
\rm{The invariant $\TC_{G,r}(X)$ also has an interpretation in terms of the motion planning problem for any mechanical system whose configuration space $X$ exhibits symmetries encoded by a topological group $G$. In particular, sequential equivariant topological complexity gives a numerical measure of the difficulty in constructing sequential motion planning algorithms that preserve the symmetries of the corresponding configuration spaces.}
\end{remark}

In~\Cref{subsec: eq dtc of products of spheres}, we will estimate (and determine) $\TC_{\Z_2,r}$ of products of spheres under various $\Z_2$-actions.

\section{Distributional sectional category of fibrations}\label{sect: dsecat}
The notion of distributional sectional category was introduced implicitly by Dranishnikov and Jauhari in~\cite{D-J}, and then developed by Jauhari in~\cite{Jau1}. Independently, a related notion of ``analog sectional category” was studied by Knudsen and Weinberger in~\cite{K-W}. To understand this notion, we recall some notations.

For any metric space $X$, let $\B(X)$ denote the set of probability measures on $X$, and for any $k\ge 1$, let $\B_k(X)$ be the space of measures of support at most $k$, endowed with the Lévy--Prokhorov metric topology~\cite{Pr} (see also~\cite[Section 3.1]{D-J}). In the literature, $\B_k(X)$ is also studied with a more general ``quotient topology'' that comes from the topology of the \emph{symmetric join} of $k$-copies of $X$, see, for example,~\cite{K-K} and~\cite{K-W}. In this paper, we express a measure $\mu\in \B_k(X)$ as a formal \textit{unordered} linear combination
\[
\mu=\sum_{i\in K}a_ix_i,
\]
where $K$ is an indexing set with $|K|\le k$, $x_i\in X$, and $a_i\ge 0$ with $\sum_{i\in K}a_i=1$. The support of the measure $\mu$, denoted $\supp(\mu)$, is the set $\{x_i\in X\mid a_i>0 \text{ for }i\in K\}$. When the size restriction on $\supp(\mu)$ is clear, we shall omit the subscript $i\in K$ in the summation.

Let $p \colon E\to B$ be a map, where $E$ is a metric space. For each $k \ge 1$, let us define a space
\[
E_k(p) = \bigcup_{x \in B}\mathcal{B}_k  (p^{-1}(x)) = \left\{ \mu \in \B_k(E) \hspace{1mm} \middle| \hspace{1mm} \supp(\mu) \subset p^{-1}(x), \hspace{1mm} x \in B \right\}
\]
and a map $\mathcal{B}_k(p) \colon E_k(p) \to B$ such that
\[
\mathcal{B}_k(p)(\mu) = x \ \text{ whenever } \ \mu \in \mathcal{B}_k (p^{-1}(x)). 
\]
This map is a fibration if $p$ is a fibration, see~\cite[Proposition 5.1]{D-J}. We now recall the definition of the following invariant of probabilistic flavor.

\begin{definition}[\protect{\cite[Definition 5.2]{Jau1}}]\label{def: dsecat}
    \rm{Let $p\colon E\to B$ be a fibration, where $E$ is a metric space. The \emph{distributional sectional category} of $p$, denoted $\dsct(p)$, is defined as the smallest integer $n$ such that the fibration $\mathcal{B}_{n+1}(p) \colon E_{n+1}(p) \to B$ admits a section.}
\end{definition} 

\begin{remark}\label{rmk: dsecat vs secat}
\rm{When $B$ is paracompact, this definition is the distributional analogue of the equivalent definition of $\secat$ that follows from Proposition~\ref{prop: sectional category}. We also see from Proposition~\ref{prop: sectional category} that $\dsct(p)\le \secat(p)$ for any fibration $p\colon E\to B$. This is because for any fibration $p\colon E\to B$ and integer $k\ge 1$, there is a map $\ast^kE\to \B_k(E)$ from the $k$-th iterated join of $E$ to the space of probability measures of support at most $k$ on $E$, which descends to a map $\ast^k_B E\to E_k(p)$ between the respective fiberwise subspaces. For technical details about these maps, we refer to~\cite[Theorem 1.1]{K-K},~\cite[Section 2]{K-W}, and~\cite[Section 8]{Jau1}.}
\end{remark}

The invariant $\dsecat$ satisfies most of the nice properties satisfied by $\secat$ as mentioned in~\Cref{subsec: secat}. An analogue of  Property~\ref{(1)} for $\dsecat$ (i.e., the homotopy invariance of $\dsecat$) was shown in~\cite[Proposition 5.3]{Jau1} as follows. 

\begin{proposition}\label{prop: homo inv of dsecat}
    If $p\colon E\to B$ and $p'\colon E'\to B'$ are fibrations in the commutative diagram
\[
 \begin{tikzcd}[every arrow/.append style={shift left}]
E \arrow{r}{f} \arrow[swap]{d}{p}
& 
E'  
\arrow{d}{p'}  
\\
B \arrow{r}{g}
& 
B',
\end{tikzcd}
\]
where $f$ and $g$ are homotopy equivalences, then $\dsecat(p)=\dsecat(p')$.
\end{proposition}

Here, we prove some other natural properties of $\dsecat$ (the next two also appear in~\cite{J-O}).

\begin{proposition}\label{prop: pullback ineq dsecat}
Suppose the following diagram is a pullback  
\[\begin{tikzcd}
    E' \arrow[r, ""] \arrow[d, "p'"'] & E \arrow[d, "p"] \\
    B' \arrow[r, "f"]                  & B.       \end{tikzcd}\]
Then  $\dsct(p')\leq \dsct(p)$.  
\end{proposition}
\begin{proof}
Let $\dsecat(p)=n$. Note that we have a fibration $\mathcal{B}_{n+1}(p)\colon E_{n+1}(p)\to B$. Then we have the following description of the pullback of  $\mathcal{B}_{n+1}(p)$ along $f\colon B'\to B$:
\begin{align*}
f^*E_{n+1}(p)&=\{(b',\mu)\in B'\times E_{n+1}(p) \mid \supp(\mu)\subseteq p^{-1}(f(b')) \text{ for some } b'\in B'\}    \\
&= \bigcup_{b'\in B'}\B_{n+1}(p^{-1}(f(b')))\\
&\cong \bigcup_{b'\in B'}\B_{n+1}((p')^{-1}(b')) \ \text{ (since $(p')^{-1}(b')\cong p^{-1}(f(b'))$)}\\
&=E'_{n+1}(p').
\end{align*}
Therefore, the fibration $\B_{n+1}(p')$ is the pullback of the fibration $\B_{n+1}(p)$ along the map $f$, with the pullback space being $E'_{n+1}(p')$. Suppose $s\colon B\to E_{n+1}(p)$ is a section of $\mathcal{B}_{n+1}(p)$. Then, using  the universal property of pullbacks as shown in the following diagram:
 \[\begin{tikzcd}
 B'\arrow[ddr, bend right, "\text{Id}_{B'}"] \arrow[drr, bend left, "s\circ f"] \arrow[dr, dashed, " s'"] \\
 &
 E'_{n+1}(p') \arrow{d}{\mathcal{B}_{n+1}(p')} \arrow{r}{} 
 & E_{n+1}(p) \arrow{d}{\mathcal{B}_{n+1}(p)} 
 \\
 &
 B' \arrow{r}{f} 
 & 
 B,
\end{tikzcd}\]
 we obtain a section $s'\colon B'\to E'_{n+1}(p')$ of $\mathcal{B}_{n+1}(p')$. This implies $\dsecat(p')\le n$, hence giving us the desired inequality.
\end{proof}

\begin{proposition}\label{prop: composition ineq dsecat}
Suppose $p\colon E\to B$ and $p'\colon E'\to B$ are fibrations satisfying the following homotopy commutative diagram:
\[\begin{tikzcd}
E \arrow[dr, "p"'] \arrow{rr}{f}
& & E' \arrow{dl}{p'} \\
& B.
\end{tikzcd}\] 
Then $\dsct(p')\leq \dsct(p)$.    
\end{proposition}
\begin{proof}
Let $\dsecat(p)=n$. Note that we have the following homotopy commutative diagram: 
\[\begin{tikzcd}
E_{n+1}(p) \arrow[dr, "\mathcal{B}_{n+1}(p)"'] \arrow{rr}{f_{*}}
& & E'_{n+1}(p) \arrow{dl}{\mathcal{B}_{n+1}(p')} \\
& B,
\end{tikzcd}\] 
where $f_*$ is defined using the functoriality of $\mathcal{B}_{n+1}$. Suppose $s\colon B\to E_{n+1}(p)$ is a section of $\mathcal{B}_{n+1}(p)$. Then $f_*\circ s$ defines a homotopy section of $\mathcal{B}_{n+1}(p')$. Since $\mathcal{B}_{n+1}(p')$ is a fibration, we obtain a section of $\mathcal{B}_{n+1}(p')$. Thus, $\dsecat(p')\le n$.
\end{proof}

Two particular distributional invariants of interest, namely the distributional category and sequential distributional complexity, are recovered from the definition of $\dsecat$ in precisely the same way as the classical invariants LS-category and sequential topological complexity are recovered from the alternative definition of $\secat$ (\emph{cf.}~\Cref{prop: sectional category}).

Recall the evaluation fibration $p^X\colon P_0(X)\to X$ from~\Cref{subsec: secat}. Then the \emph{distributional category} of $X$, denoted $\dcat(X)$, is defined as $\dcat(X):=\dsecat(p^X)$, see~\cite{D-J},~\cite{K-W}. 

Recall also the path fibration $\pi_r^X\colon X^I\to X$ for each $r\ge 2$ from~\eqref{eq: sequential free path fibration}. Then for each $r\ge 2$, the $r$\emph{-th sequential distributional complexity} of $X$, denoted $\dTC_r(X)$, is defined as $\dTC_r(X):=\dsecat(\pi_r^X)$, see~\cite{Jau1},~\cite{K-W} (and also~\cite{D-J} for the case $r=2$). 

Like $\TC_r$, the invariant $\dTC_r$ also has interpretations in terms of motion planning. We refer to the introduction for that and for the motivation behind this invariant.

For any metric space $X$, the distributional invariants satisfy the following properties~(\cite{Jau1}):
\begin{enumerate}
    \item $\dcat(X)\le\cat(X)$ and $\dTC_r(X)\le\TC_r(X)$ for each $r\ge 2$.\label{one}
    \item $\dcat(X^{r-1})\le\dTC_r(X)\le\dcat(X^r)$ for each $r\ge 2$. The lower bound is attained when $X$ is a CW $H$-space (in particular, a compact Lie group).\label{two}
    \item $\dcat(\R P^n)=1$ and $\dTC_r(\R P^n)\le \min\{2r+1,2^{r-1}-1\}$ for all $n\ge 1$ and $r\ge 2$, so that the inequalities in Property~\ref{one} can be strict.\label{three}
    \item $\dTC_r(X)=0$ for any $r\ge 2$ if and only if $X$ is contractible.\label{four}
\end{enumerate}

Distributional invariants are most certainly different from the classical ones. Besides $\R P^n$, they also differ on the classifying spaces of finite groups (see~\cite{K-W},~\cite{Dr},~\cite{K-W2}), and on lens spaces and projective product spaces (see~\cite{J-O}).

\section{Sequential parametrized distributional complexity}\label{sec: sequential parametrized distributional complexity}

In this section, we introduce and develop a distributional analogue of the notion of sequential parametrized topological complexity reviewed in~\Cref{subsect: parametrized TC}. We refer to~\Cref{intro: parametrized} for interpretations of this distributional invariant in terms of motion planning.

Let $p\colon E\to B$ be a fibration, where $E$ is a metric space. We recall from~\eqref{eq: parametrized endpoint map} that for $r\ge 2$, the fibration $\Pi_r^E\colon E^I_B\to E^r_B$ evaluates each path $\gamma\in E^I_B$ at the points $\tfrac{i-1}{r-1}$ for $1\le i \le r$.

\begin{definition}\label{def: main def}
\rm{For $r\ge 2$, the $r$-\emph{th sequential parametrized distributional complexity} of a fibration $p\colon E\to B$, denoted $\dTC_r[p\colon E\to B]$, is defined as the distributional sectional category of the fibration $\Pi_r^E$, i.e., $\dTC_r[p\colon E\to B]:=\dsct(\Pi_r^E)$.  }
\end{definition}

We note that the inequality $\dTC_r[p\colon E\to B]\le \TC_r[p\colon E\to B]$ holds if $B$ is paracompact (\emph{cf.}~\Cref{sect: dsecat}).

Let us now interpret $\dTC_r[p\colon E\to B]$ in terms of a navigation algorithm. For any $(e_1,\dots, e_r)\in E^r_B$, we define  $E^I_B(e_1,\dots,e_r)$ to be subspace of all paths $\gamma$ such that $\gamma(\tfrac{i-1}{r-1})=e_i$ for $1\le i\le r$, which lie in a fiber $p^{-1}(b)$, where $b=p(e_i)\in B$ for all $i$.
\begin{definition}\label{def: algorithm}
\rm{The $r$-\emph{th sequential parametrized $k$-distributed navigation algorithm} for a fibration $p\colon E\to B$ is a map 
$$
s\colon E^r_B \to \mathcal{B}_k(E^I_B)
$$
satisfying $s(e_1,\dots,e_r)\in \mathcal{B}_k(E^I_B(e_1,\dots,e_r))$ for all $(e_1,\dots, e_r)\in E^r_B$.}
\end{definition}
In other words, such an algorithm maps each tuple $(e_1,\dots,e_r)\in E^r_B$ to a probability measure of at most $k$ paths $\gamma_j\in E^I_B$, where for each $j$, $\gamma_j(\tfrac{i-1}{r-1})=e_i$ for all $1\le i\le r$. We note that in contrast, an $r$-th sequential parametrized motion planning algorithm (\emph{cf.}~\Cref{subsect: parametrized TC}) maps each $(e_1,\dots,e_r)$ to an \emph{ordered} linear combination of at most $k$ paths $\gamma_j$ satisfying the same evaluation property as above. 

It follows from~\cref{def: main def} that $\dTC_r[p\colon E\to B]$ is the smallest integer $n$ such that $p$ admits an $r$-th sequential parametrized $(n+1)$-distributed navigation algorithm.


\begin{remark}\label{rmk: dtc same as dptc}
\rm{Observe that if $B=\{\ast\}$, then $E^I_B=E^I$ and $E^r_B=E^r$, so that the fibration $\Pi^E_r$ coincides with the (non-parametrized) evaluation fibration $\pi_r^E\colon E^I\to E^r$ defined in~\eqref{eq: sequential free path fibration}. Hence, $\dTC_r[p\colon E\to \ast]$ is just the $r$-th sequential distributional complexity $\dTC_r(E)$.}
\end{remark}

So, the notion of sequential parametrized distributional complexity of fibrations generalizes the notion of sequential distributional complexity of spaces.

\subsection{Some standard homotopy-theoretic properties}\label{subsec: basic properties}
In this section, we prove some basic properties and inequalities for sequential parametrized distributional complexities.

We first recall the definition of fiber-preserving homotopy invariance.  

\begin{definition}\label{def: fib hteq}
\rm{Let $p\colon E\to B$ and $p'\colon E'\to B'$ be two fibrations. A \emph{fiber-preserving map} from $p$ to $p'$ is a pair of maps $f\colon E\to E'$ and $g\colon B\to B'$ satisfying $p'\circ f=g\circ p$, denoted $(f,g)\colon p \to p'$; a \emph{fiber-preserving homotopy} is a pair of maps $F\colon E\times I\to E'$ and $G\colon B\times I\to B'$ such that $p'\circ F=G\circ (p\times\text{Id}_I)$, denoted $(F,G)$; and fibrations $p$ and $p'$ are \emph{fiber-preserving homotopy equivalent} if there are fiber-preserving maps $(f,g)\colon p\to p'$ and $(f',g')\colon p'\to p$ such that $(f\circ f',g\circ g')$ is fiber-preserving homotopic to $(\text{Id}_{E'},\text{Id}_{B'})$ and $(f'\circ f,g'\circ g)$ is fiber-preserving homotopic to $(\text{Id}_{E},\text{Id}_{B})$.}
\end{definition}

We now establish the fiber-preserving homotopy invariance of the sequential parametrized distributional complexity along the lines of~\cite{MarkGrant}.

\begin{proposition}\label{prop: tc homo inv}
If $p\colon E\to B$ and $p'\colon E'\to B'$ are fibrations in the commutative diagram
\[
 \begin{tikzcd}[every arrow/.append style={shift left}]
E \arrow{r}{f} \arrow[swap]{d}{p}
& 
E'  
\arrow{d}{p'}  
\\
B \arrow{r}{g}
& 
B',
\end{tikzcd}
\]
where $f$ and $g$ are homotopy equivalences, then $\dTC_r[p\colon E\to B]=\dTC_r[p'\colon E'\to B'].$
\end{proposition}
\begin{proof}
First, we note that $(f,g)\colon p\to p'$ is a fiber-preserving map. Suppose $(F,G)$ is a fiber-preserving homotopy. Then for the fibrations $\Pi_r^E\times\text{Id}_I\colon E^I_B\times I\to E^r_B\times I$ and $\Pi_r^{E'}\colon (E')^I_{B'}\to (E')^r_{B'}$, we get a fiber-preserving homotopy $(A,B)$ induced by $(F,G)$. In particular, there are homotopy equivalences $a\colon E^I_B\to (E')^I_{B'}$ and $b\colon E^r_B\to (E')^r_{B'}$ such that $\Pi_r^{E'}\circ a=b\circ\Pi^E_r$ (we refer to~\cite[Proposition 2.6]{MarkGrant} for details). Since $a$ and $b$ are homotopy equivalences, we get by~\Cref{prop: homo inv of dsecat} that
\[
\dsecat(\Pi_r^{E'})=\dsecat\left(\Pi_r^{E'}\circ a\right) = \dsecat\left(b\circ \Pi_r^{E}\right)=\dsecat(\Pi_r^{E}).
\]
This completes the proof.
\end{proof}

\begin{proposition}\label{prop: fiber lower bound}
    Suppose $p \colon E \to B$ is a fibration and $B'\subseteq B$. 
    If $E' = p^{-1}(B')$ and $p' \colon E' \to B'$ is the  restricted fibration, then 
    \begin{equation}\label{eq: fiber ineq}
     \dTC_r[p' \colon E' \to B'] \leq \dTC_r[p \colon E \to B].   
    \end{equation}
    In particular, if $B'=\{b\}$ and $F:=p^{-1}(b)$ is the fiber of $p$, then
    $$
        \dTC_r(F) \leq \dTC_r[p \colon E \to B].
    $$    
\end{proposition}
\begin{proof}
Suppose $\dTC_r[p \colon E \to B]=k-1$, so that $s\colon E^r_B\to \mathcal{B}_k(E^I_B)$ is an $r$-th sequential parametrized $k$-distributed navigation algorithm for the fibration  $p\colon E\to B$. 
Define $s'=s\circ i\colon (E')^r_{B'}\to \mathcal{B}_k(E^I_B)$, where $i\colon (E')^r_{B'} \hookrightarrow{} E^r_B$ is the inclusion map. For any fixed $(e_1,\dots, e_r)\in (E')^r_{B'}$, since $s'(e_1,\dots,e_r)\in \mathcal{B}_{k}(E^I_{B})$, we can write 
\[
s'(e_1,\dots,e_r)=\sum\lambda_{\phi}\phi,
\]
where $\phi\in E^I_B(e_1,\dots,e_r)$, i.e., $\phi(\frac{i-1}{r-1})=e_{i}\in E'$ for $1\leq i\leq r$ and $p\circ \phi(t)=b$ for some $b\in B$, and for all $t\in I$. 
Note that $b=p\circ \phi(t)=p\circ \phi(0)=p(e_1)$ for each $t\in I$. But we have $p(e_1)=p'(e_1)\in B'$ by definition. This implies $b\in B'$ and therefore, $\phi(t)\in E'$ for all $t\in I$. This shows that $\phi\in (E')^I_{B'}(e_1,\dots,e_r)$, and so $s'(e_1,\dots,e_r)\in \mathcal{B}_{k}((E')^I_{B'}(e_1,\dots,e_r))$ as desired. Hence, $\dTC_r[p'\colon E'\to B']\le k-1$.

Now, if $B'=\{b\}$, then $E'=p^{-1}(b)$. Then the inequality $ \dTC_r(F) \leq \dTC_r[p \colon E \to B]$ follows from~\eqref{eq: fiber ineq} and~\Cref{rmk: dtc same as dptc}.
\end{proof}

\begin{corollary}\label{cor: rationally acyclic}
    If $p\colon E\to B$ is a fibration whose fiber $F$ is a finite CW complex that is not rationally acyclic, then $\dTC_r[p\colon E\to B]\ge r-1$.
\end{corollary}
\begin{proof}
    We have from~\Cref{prop: fiber lower bound} that $\dTC_r[p\colon E\to B]\ge \dTC_r(F)$. Since $F$ is not rationally acyclic, there exists $d\ge 1$ such that $H^d(F;\Q)\ne 0$. Then~\cite[Proposition 7.1]{Jau1} implies that $\dTC_r(F)\ge r-1$. 
\end{proof}

The lower bound from~\Cref{prop: fiber lower bound} can be attained for trivial fibrations, as shown in the following proposition.

\begin{proposition}
If $p\colon E\to B$ is a trivial fibration with fiber $F$, then we have for each $r\ge 2$ that $\dTC_r[p:E\to B]=\dTC_r(F)$.    
\end{proposition}
\begin{proof}
Since $p\colon E\to B$ is a trivial fibration, we can assume that $E=F\times B$ and $p$ is the projection. Then $E^r_B\cong F^r\times B$, $E^I_B\cong F^I\times B$, and the fibration $\Pi_r^E$ coincides with $\pi_r^F\times \text{Id}_B$, where $\pi_r^F\colon F^I\to F^r$ is the free path fibration from~\eqref{eq: sequential free path fibration}.
Then we have 
$$
\dTC_r[p:E\to B]=\dsct(\Pi_r^E)=\dsct(\pi_r^F\times \text{Id}_B)=\dsct(\pi_r^F)=\dTC_r(F),
$$ 
where the third equality holds since $(F^I\times B)_k(\pi^F_r\times\text{Id}_B)\cong (F^I)_k(\pi^F_r)\times B$ for all $k$.
\end{proof}

Next, we observe that the contractibility of the fiber of a fibration determines whether its sequential parametrized distributional complexities are minimal.

\begin{proposition}\label{prop: fiber contract}
If $p\colon E\to B$ is a fibration with fiber $F$ and $\dTC_r[p\colon E\to B]=0$, then $F$ is contractible. Conversely, if $F$ is contractible and $E^r_B$ has the homotopy type of a CW complex, then $\dTC_r[p\colon E\to B]=0$.
\end{proposition}
\begin{proof}
From~\Cref{prop: fiber lower bound}, it follows that $\dTC_r(F)\leq \dTC_r[p\colon E\to B]$. Therefore, if $\dTC_r[p\colon E\to B]=0$, then $\dTC_r(F)=0$. This implies $F$ is contractible (\emph{cf.} Property~\ref{four} in~\Cref{sect: dsecat}). Conversely, if $F$ is contractible, then for any $r\ge 2$, the fiber $(\Omega F)^{r-1}$ of the fibration $\Pi_r^E\colon E^I_B\to E^r_B$ is also contractible. Hence, $\Pi_r^E$ is a weak homotopy equivalence in view of the long exact sequence of homotopy groups. Now, there exists a CW complex $E'$ and a weak homotopy equivalence $f\colon E'\to E$, see~\cite[Proposition~16.14]{Gray}. We also have a CW complex $E''$ and homotopy equivalences $g\colon E^r_B\to E''$ and $e\colon E\to E^I_B$. Thus, the composition $\theta=g\circ \Pi^E_r\circ e\circ f\colon E'\to E''$ is a weak homotopy equivalence of CW complexes and hence a homotopy equivalence by Whitehead's theorem. Using a homotopy inverse $\theta^{-1}\colon E''\to E'$, 
we get a homotopy section $e\circ f \circ \theta^{-1}\circ g\colon E^r_B\to E^I_B$ of the fibration $\Pi^E_r$, and so a section. 
Therefore, $\dTC_r[p\colon E\to B]=\dsecat(\Pi_r^E)=0$. 
\end{proof}

\begin{remark}
    {\rm There is a more direct (but somewhat less elementary) proof of the converse in Proposition~\ref{prop: fiber contract}. Indeed, using the contractibility of the fiber $(\Omega F)^{r-1}$ of $\Pi^E_r$, we can apply obstruction theory to construct a global section of $\Pi_r^E$, see~\cite{Steenrod} or~\cite[Chapter 4]{Hatcher}. }
\end{remark}

\subsection{Computations for principal bundles}
In the classical case (see~\cite{C-F-W},~\cite{F-P1}), there is the equality $\TC_r[p\colon E\to B]=\TC_r(G)$ if $p$ is a principal $G$-bundle. We have an analogous behavior in the distributional setting as well.

\begin{proposition}\label{prop: seq dtc of principal G bundles}
Suppose $p\colon E\to B$ is a principal $G$-bundle. Then 
$$
\dTC_r[p\colon E\to B]=\dTC_r(G)=\dcat(G^{r-1}).
$$
\end{proposition}
\begin{proof}
Let $e\in G$ be the identity. Suppose $P_\ast(G)=\{\gamma\in P(G)\mid\gamma(0)=e\}$ is the based path space of $G$. Then it follows from~\cite[Proposition 5.7]{Jau1} that $\dcat(G^{r-1})=\dsct(q^G)$, where $q^G\colon P_\ast(G)\to G^{r-1}$ is defined by 
$$
q^G(\gamma)=\left(\gamma\left(\frac{1}{r-1}\right),\dots,\gamma\left(\frac{r-2}{r-1}\right),\gamma(1)\right).
$$
Our approach resembles the approach of~\cite[Proposition 3.3]{F-P1}. Consider the following commutative diagram:
\[ \begin{tikzcd}
P_\ast(G)\times E \arrow{r}{F} \arrow[swap]{d}{q^G\times \text{Id}_E} 
& 
E^I_B \arrow{d}{\Pi_r^E} 
\\
G^{r-1}\times E \arrow{r}{F'}
&
E^r_B,
\end{tikzcd}
\]
where $F$ and $F'$ are homeomorphisms defined by  
\[
F(\gamma,x)(t)=x \cdot \gamma(t) ~\text{ and }~  F'(a_1,\dots,a_{n-1},x)=(x,x\cdot a_1,\dots,x\cdot a_{n-1}),
\]
respectively. Therefore, using~\Cref{prop: homo inv of dsecat}, we obtain
$$
\dTC_r[p\colon E\to B]=\dsct(\Pi_r^E)=\dsct(q^G\times \text{Id}_{E})=\dsct(q^G)=\dcat(G^{r-1}).
$$
The equality $\dTC_r(G)=\dcat(G^{r-1})$ follows from Property~\ref{two} in~\Cref{sect: dsecat}.
\end{proof}

\Cref{prop: seq dtc of principal G bundles} gives examples of non-trivial bundles for which the inequality of~\Cref{prop: fiber lower bound} is an equality.

\begin{example}
 \rm{Suppose $p\colon S^{2n+1}\to \C P^n$ denotes the Hopf bundle, which is a principal $S^1$-bundle. Then we have  $\dTC_r[p\colon S^{2n+1}\to \C P^n]=\dTC_r(S^1)=r-1$ from~\Cref{prop: seq dtc of principal G bundles}. For the last equality, see~\cite[Proposition 7.8]{Jau1}.}
\end{example}

We now provide examples of fibrations whose sequential parametrized distributional complexities are strictly smaller than their corresponding sequential parametrized topological complexities.
\begin{example}\label{ex: temporary}
 \rm{Recall that for a principal $G$-bundle $p\colon E\to B$ and integer $r\ge 2$, we have the equalities 
 \[
 \dTC_r[p\colon E\to B]=\dTC_r(G)=\dcat(G^{r-1})   \ \text{ and } \ \TC_r[p\colon E\to B]=\TC_r(G)=\cat(G^{r-1}).
 \]
Also recall that $\cat((\R P^n)^r)=nr$ (see~\cite{James},~\cite{CLOT}) and $\dTC_r(\R P^n)\le \min\{2^{r-1}-1,2r+1\}$, see Property~\ref{three} in~\Cref{sect: dsecat}. So, it suffices to find principal $G$-bundles for which $G\cong\R P^n$ for $n\geq 2$. It is well known that the special orthogonal group $SO(3)$ is homeomorphic to $\R P^3$. Thus, all principal $SO(3)$-bundles give us our desired set of examples. In fact, for a principal $SO(3)$-bundle $p\colon E\to B$, we have for $2\le r\le 4$ that
 \[
\dTC_r[p\colon E\to B]\le 2^{r-1}-1<3(r-1)=\TC_r[p\colon E\to B],
 \]
 and for all $r\ge 5$ that
\[
\dTC_r[p\colon E\to B]\le 2r+1<3(r-1)=\TC_r[p\colon E\to B].
\]
We note that there are infinitely many such principal $SO(3)$-bundles since the set of non-trivial principal $SO(3)$-bundles over the $4$-sphere $S^4$ is in one-to-one correspondence with the set of integers $\Z$ --- see~\cite[Theorem 18.5]{Steenrod} for more details (and also~\Cref{rmk: bott}).}
\end{example}

In~\Cref{prop: different examples}, we will see several examples of non-principal bundles whose sequential parametrized distributional complexities are arbitrarily smaller than their corresponding sequential parametrized topological complexities.

\section{Cohomological lower bounds}\label{sec: lower bounds}
The goal of this section is to obtain a cohomological lower bound to $\dTC_r[p\colon E\to B]$ for any fibration $p\colon E\to B$, where $E$, $B$, and the fiber $F$ have the homotopy type of CW complexes. 

Before stating our result, we recall some notations. The $k$\emph{-th symmetric product} of a topological space $X$, denoted $SP^k(X)$, is the orbit space of the Cartesian product $X^k$ under the natural permutation action of the symmetric group $S_k$ on $k$-symbols. The elements of $SP^k(X)$ are orbits of elements $(x_1,\dots,x_k)\in X^k$, and any such orbit can be considered as a formal sum $\sum k_ix_i$, where $\sum k_i=k$, subject to the equivalence $(n+\ell)x=nx+\ell x$. 

Let $\Delta_r\colon E\to E^r_B$ be the diagonal map, and for any integer $k$, let $\delta_k\colon E^r_B\to SP^k(E^r_B)$ be the diagonal embedding, which we will refer to as \emph{inclusion} in this section and often write simply as $E^r_B\hookrightarrow SP^k(E^r_B)$. Note that in the formal sum notation described above, we have $\delta_k(e_1,\dots,e_r)=k(e_1,\dots,e_r)$.

Our main result of this section can be stated as follows in Alexander--Spanier cohomology. 

\begin{theorem}\label{th: lower bound}
    For fixed integers $r\ge 2$ and $n\ge 1$ and ring $R$, suppose $\alpha_i^*\in H^{k_i}(SP^{n!}(E^r_B);R)$ for $1\le i \le n$, where $k_i\ge 1$. Suppose $\delta_{n!}^*(\alpha_i^*)=\alpha_i \in H^{k_i}(E^r_B;R)$ such that $\Delta_r^*(\alpha_i)=0$ for all $i$. If $\alpha_1\alpha_2\cdots\alpha_n\ne 0$, then
    \[
    \dTC_r\left[p\colon E\to B\right]\ge n.
    \]
\end{theorem}

In other words, we have for each $n\ge 1$, $r\ge 2$, and ring $R$ that 
\[
c\ell_R\left(\text{Im}(\delta_{n!}^*)\cap \text{Ker}(\Delta_r^*)\right)\ge n\implies \dTC_r[p\colon E\to B]\ge n.
\]
Here, $c\ell_R$ denotes the \emph{cup-length} (i.e., the length of the longest non-trivial products of elements) of a cohomology (sub)ring with coefficients in a ring $R$. For comparison, we recall that in the classical setting, one has a simpler lower bound $c\ell_R(\text{Ker}(\Delta_r^*))\le \TC_r[p\colon E\to B]$, see~\cite[Proposition 6.3]{F-P1} (and also~\cite{C-F-W} for the case $r=2$).

To prove~\Cref{th: lower bound}, we need an auxiliary result. For that, we recall that a set $A\subset X$ is said to \emph{deform} to a set $B\subset X$ if there exists a map $H\colon X\times I\to X$ such that $H(a,0)=a$ and $H(a,1)\in B$ for all $a\in A$.

\begin{lemma}\label{lem:lower bound prep}
    Suppose $\dTC_r[p\colon E\to B]<n$. Then there exist sets $A_1$, $A_2,\ldots,A_n$ that cover $E^r_B$ such that each $A_i$ deforms to the diagonal $\Delta_r E\hookrightarrow E^r_B\hookrightarrow SP^{n!}(E^r_B)$ in $SP^{n!}(E^r_B)$.
\end{lemma}

\begin{proof}
    Since $\dTC_r[p\colon E\to B]<n$, there exists a section $s\colon E^r_B\to (E^I_B)_n$ to the fibration $\B_n(\Pi_r^B)\colon (E^I_B)_n\to E^r_B$ by~\Cref{def: main def}. For $1\le i \le n$, define sets
    \[
    A_i=\left\{\left(e_1,\ldots,e_r\right)\in E^r_B \ \middle| \ \left|\supp\left(s\left(e_1,\ldots,e_r\right)\right)\right|=i\right\}.
    \]
Since each $A_i$ contains tuples which are mapped under $s$ to a measure of support of exactly $i$, we have $\bigcup_{i=1}^nA_i=E^r_B$. For each $i$, we define a homotopy $H_i\colon A_i\times I\to SP^{n!}(E^r_B)$ as follows. For $(e_1,\ldots,e_r)\in A_i$, suppose 
\[
s\left(e_1,\ldots,e_r\right)=\sum_{\phi\in \supp\left(s\left(e_1,\ldots,e_r\right)\right)} \lambda_{\phi}\phi.
\]
Then for each $1 \le j \le r$, define paths $\phi_j\in E^I_B$ as
\[
\phi_j(t)=\phi\left(\frac{t(r-j)+j-1}{r-1}\right).
\]
Each $\phi_j$ starts at $e_j$ and terminates at $e_r$. In particular, $\phi_1=\phi$ and $\phi_r$ is the constant path at $e_r$. Then, we define
\[
H_i\left(e_1,\ldots,e_r,t\right)=\sum_{\phi\in \supp\left(s\left(e_1,\ldots,e_r\right)\right)} \frac{n!}{i}\left(\phi_1(t),\phi_2(t),\ldots,\phi_r(t)\right).
\]
It is clear that $H_i(e_1,\ldots,e_r,0)=n!(e_1,\ldots,e_r)$ and $H_i(e_1,\ldots,e_r,1)=n!(e_r,\ldots,e_r)$ for all $r$-tuples in $A_i$. Hence, $H_i$ is the required deformation.
\end{proof}

\begin{proof}[Proof of~\Cref{th: lower bound}]
We proceed by contradiction. Suppose $\dTC_r[p\colon E\to B]<n$. Then by Lemma~\ref{lem:lower bound prep}, we have sets $A_1,\ldots,A_n$, each of which deform to $\Delta_r E$ in $SP^{n!}(E^r_B)$ such that $E^r_B=\bigcup_{i=1}^nA_i$. Let us fix some $i$. We denote the inclusion $A_i\hookrightarrow E^r_B\hookrightarrow SP^{n!}(E^r_B)$ by $\pa_i$. For the pair $(SP^{n!}(E^r_B),A_i)$, we have a long exact sequence~\cite{Sp},
    \[
    \begin{tikzcd}
        \cdots \arrow{r} 
        &
        H^{k_i}(SP^{n!}(E^r_B),A_i;R) \arrow{r}{j_i^*}
        &
        H^{k_i}(SP^{n!}(E^r_B);R)\arrow{r}{\pa_i^*} 
        &
        H^{k_i}(A_i;R)\arrow{r} 
        &
        \cdots,
    \end{tikzcd}
    \]
where $j_i$ is the inclusion. Since $\Delta_r^*(\alpha_i)=0$ and $A_i$ deforms to $\Delta_r E$ inside $SP^{n!}(E^r_B)$, we have $\pa_i^*(\alpha_i^*)=0$, so that there exists a class $\beta_i^*\in H^{k_i}(SP^{n!}(E^r_B),A_i;R)$ satisfying $j_i^*(\beta_i^*)=\alpha_i^*$. Suppose $\beta_i\in H^k(E^r_B,A_i;R)$ is the image of $\beta_i^*$ under the natural map. Then for $k=\sum_{i=1}^nk_i$, we have the commutative diagram
\[
\begin{tikzcd}
    H^k(SP^{n!}(E^r_B),\bigcup_{i=1}^nA_i;R) \arrow{r}{j}  \arrow{d}{\delta_{n!}^*}
    &
    H^k(SP^{n!}(E^r_B);R) \arrow{d}{\delta_{n!}^*}
    \\
    H^k(E^r_B,\bigcup_{i=1}^nA_i;R) \arrow{r}{}  
    &
    H^k(E^r_B;R), 
\end{tikzcd}
\]
where $j$ is the sum of the maps $j_i$. Then, the cup-product $\beta_1^*\beta_2^*\cdots\beta_n^*$ in the top-left group maps to $\alpha_1\alpha_2\cdots\alpha_n\ne 0$ in the bottom-right group. But on the other hand, it factors through $\beta_1\beta_2\cdots\beta_n\in H^k(E^r_B,\bigcup_{i=1}^nA_i;R)=H^k(E^r_B,E^r_B;R)=0$ due to commutativity. This is a contradiction.
\end{proof}

\begin{remark}
    \rm{We used Alexander--Spanier cohomology in~\Cref{th: lower bound} because the sets $A_i$ in Lemma~\ref{lem:lower bound prep} need not be open or closed, and hence, $(SP^{n!}(E^r_B),A_i)$ may not be a CW pair,~\cite{Hatcher}. In the case of metric ANR spaces, however, one can still use singular cohomology as explained in~\cite{D-J}. 
    
    In detail, if $E$ is a metric ANR, then so is $SP^{n!}(E^r_B)$ --- this can be derived from~\cite{Zarichnyi}. Since $(E^r_B,\Delta E)$ and $(SP^{n!}(E^r_B),E^r_B)$ are NDR-pairs, so is $(SP^{n!}(E^r_B),\Delta E)$. We recall that $(X,Y)$ is called an \emph{NDR-pair} if $Y\hookrightarrow X$ is a cofibration. Now, as $A_i\subset SP^{n!}(E^r_B)$ admits a deformation to $\Delta E\subset SP^{n!}(E^r_B)$ due to~\Cref{lem:lower bound prep}, it can be deduced by generalizing the technique of~\cite[Lemma 2.7]{Sri1} and using a modification (as in~\cite[Theorem 2.3]{Sri2}) of the classical Walsh Lemma from~\cite{Wal} that there exists an open neighborhood $U_i\subset SP^{n!}(E^r_B)$ of $A_i$ that also deforms to $\Delta_r E$ in $SP^{n!}(E^r_B)$, see also~\cite{GC2}. Thus, in the setting of metric ANRs, one can replace the sets $A_i$ in the proof of~\Cref{th: lower bound} by \emph{open sets} $U_i$ and use the long exact sequence for the pair $(SP^{n!}(E^r_B),U_i)$ in singular cohomology to get the same lower bound.

    Nevertheless,~\Cref{th: lower bound} in its current form suffices for our purpose since we will primarily use it in~\Cref{sec: ex and comput} for computations involving (locally) finite CW complexes, for which it is well-known that the Alexander--Spanier cohomology coincides with the singular cohomology.}
\end{remark}

For the sake of completeness, we note (without proof) that another formal lower bound to the invariant $\dTC_r[p\colon E\to B]$ can be obtained in terms of the cohomology of the diagonal map $\tau_{r,k}\colon SP^{k}(E)\to (SP^{k}(E))^r_B$ for each $r\ge 2$ following ideas from~\cite{Jau2}. Here,
\[
(SP^{k}(E))^r_B:=\left\{\left([e_1^1,\ldots,e_1^k],\ldots,[e_r^1,\ldots,e_r^k]\right)\ \middle| \ p_k[e_i^1,\ldots,e_i^k]=p_k[e_j^1,\ldots,e_j^k]\text{ for all }i,j \right\},
\]
where $p_k=SP^k(p)\colon SP^k(E)\to SP^k(B)$ is the natural map induced from $p$ due to functoriality of $SP^k$. Define $\pa_{r,k}\colon E^r_B\to (SP^k(E))^r_B$ as the diagonal map
\[
\pa_{r,k}(e_1,\ldots,e_r)=(ke_1,\ldots,ke_r).
\]
Then the following can be shown in Alexander--Spanier cohomology (a proof can be obtained by proceeding as in~\cite[Section 4]{Jau2}).
\begin{theorem}
     For fixed integers $r\ge 2$ and $n\ge 1$ and ring $R$, suppose $\alpha_i^*\in H^{k_i}((SP^{n!}(E))^r_B;R)$ for $1\le i \le n$, where $k_i\ge 1$, such that $\tau_{r,n!}^*(\alpha_i^*)=0$ for all $i$. Let $\pa_{r,n!}^*(\alpha_i^*)=\alpha_i \in H^{k_i}(E^r_B;R)$. If $\alpha_1\alpha_2\cdots\alpha_n\ne 0$, then $\dTC_r\left[p\colon E\to B\right]\ge n$.
\end{theorem}
In this paper, we will find~\Cref{th: lower bound} more useful for computations.

\section{Examples and computations}\label{sec: ex and comput}

The following result will be pivotal to the discussions in this section.

\begin{proposition}[\protect{\cite{D-J}}]\label{prop: dj epi result}
    If $X$ is a finite CW complex\hspace{0.3mm}\footnote{\hspace{0.3mm}In~\cite{D-J}, this statement is given for finite \emph{simplicial} complexes. But the same statement holds for finite \emph{CW} complexes as well --- see, for instance,~\cite[Theorem 2C.5]{Hatcher}.}, then for each $k,j\ge 1$, the map $\delta_k^*\colon H^j(SP^k(X);\Q)\to H^j(X;\Q)$ induced by the diagonal embedding $\delta_k\colon X\to SP^k(X)$ in singular cohomology is a split epimorphism.
\end{proposition}

This helps us obtain the following very useful consequence of~\Cref{th: lower bound}.

\begin{proposition}\label{prop: zcl bound}
    Suppose $p\colon E\to B$ is a fibration of finite CW complexes and $\Delta_r\colon E\to E^r_B$ is the diagonal map for a fixed $r\ge 2$. Then 
\[
c\ell_{\Q}(\textup{Ker}(\Delta_r^*))\le\dTC_r[p\colon E\to B].
\]
\end{proposition}
\begin{proof}
We first note that $E^r_B$ is a finite CW complex (\emph{cf.}~\Cref{rmk: cw structure on erb}). Let $c\ell_{\Q}(\textup{Ker}(\Delta_r^*))=n$ and $u_i\in \text{Ker}(\Delta_r^*)\subset H^*(E^r_B;\Q)$ be positive-degree classes such that $u_1u_2\cdots u_n\ne 0$. Due to~\Cref{prop: dj epi result}, there exist $v_i\in H^*(SP^{n!}(E^r_B);\Q)$ such that $\delta_{n!}^*(v_i)=u_i$ for each $1\le i \le n$. In particular, $u_i\in \text{Im}(\delta_{n!}^*)\cap\text{Ker}(\Delta_r^*)$. Hence, $c\ell_{\Q}(\text{Im}(\delta_{n!}^*)\cap\text{Ker}(\Delta_r^*))\ge n$. The desired inequality $\dTC_r[p\colon E\to B]\ge n$ now follows from~\Cref{th: lower bound}.
\end{proof}

\subsection{Calculations for Fadell--Neuwirth fibrations on Euclidean configuration spaces}\label{subsec: fadell-neuwirth}
For any $k\ge 2$, we denote the configuration space of $k$ distinct \emph{ordered} points inside a topological space $X$ by
\[
F(X,k):=\left\{(x_1,\ldots,x_k) \in X^k\ \middle| \ x_i\ne x_j \text{ for all } i \ne j\right\}.
\]
Suppose $X$ is a manifold of dimension $d\ge 2$ without boundary. Then for integers $m\ge 2$ and $n\ge 1$, the map $p\colon F(X,m+n)\to F(X,m)$ defined as
\[
p\left(x_1,\ldots,x_m,x_{m+1},\ldots,x_{m+n}\right)=\left(x_1,\ldots,x_m\right)
\]
is a fibration, see~\cite{F-N},~\cite[Theorem 1.1]{F-H}. This fibration is called the \emph{Fadell--Neuwirth fibration} (or the \emph{Fadell--Neuwirth bundle}). For a configuration $\mathcal{O}_m=\{o_1,\ldots,o_m\}$ of any $m$ points in $X$, the fiber of $p$ over $\mathcal{O}_m$ is the space $F(X\setminus \mathcal{O}_m,n)$. 

Let us focus on the simple case when $X=\R^d$ is the Euclidean space of dimension $d\ge 2$. In this setting, the ordered configuration space $F(\R^d,k)$ has the homotopy type of a finite CW complex of dimension $(k-1)(d-1)$, and the fibers of the Fadell--Neuwirth fibration have homotopy dimension $n(d-1)$ and are $(d-2)$-connected, see~\cite[Theorem 3.13]{B-Z}. 

In this subsection, we prove the following.

\begin{proposition}\label{prop: f-n fibration}
For any integers $r,m,d\ge 2$ and $n\ge 1$, we have for the Fadell--Neuwirth fibration $p\colon F(\R^d,m+n)\to F(\R^d,m)$ that
\[
    \dTC_r[p\colon F(\R^d,m+n)\to F(\R^d,m)]=\begin{cases}
        rn+m-1 & \text{ if }d\text{ is odd} \\
        rn+m-2 & \text{ if }d\text{ is even.}
    \end{cases}
\]
In particular, $\dTC_r[p\colon F(\R^d,m+n)\to F(\R^d,m)]=\TC_r[p\colon F(\R^d,m+n)\to F(\R^d,m)]$.
\end{proposition}

The (sequential) parametrized topological complexities of these Fadell--Neuwirth fibrations were calculated in~\cite{C-F-W},~\cite{PTCcolfree},~\cite{F-P1}, and~\cite{F-P2}.

The above theorem shows that the quantity $\dTC_r[p\colon E\to B]$ can be as large as desired.

We now proceed to prepare for the proof of~\Cref{prop: f-n fibration}. For convenience, we fix $d,m\ge 2$ and $n\ge 1$ and use the notations $E:=F(\R^d,m+n)$ and $B:=F(\R^d,m)$. In these notations, our Fadell--Neuwirth fibration is $p\colon E\to B$. 

The cohomology ring $H^*(E;\Z)$ can be described as follows, see~\cite{F-H}.

\begin{theorem}
    The integral cohomology ring $H^*(E;\Z)$ is generated multiplicatively by degree $(d-1)$ classes $\omega_{ij}$ for $1\le i<j\le m+n$ that satisfy the defining relations $(\omega_{ij})^2=0$ and $\omega_{ik}\omega_{jk}=\omega_{ij}(\omega_{jk}-\omega_{ik})$ for all $i<j<k$. Moreover, $H^*(E;\Z)$ is torsion-free.
\end{theorem}

For a detailed description of classes $\omega_{ij}$, we refer to~\cite{F-H}. Since $H^*(E;\Z)$ is torsion-free, the class $\omega_{ij}$ may be considered in the rational cohomology of $E$.

We now turn towards the finite CW complex $E^r_B$. For $l\le r$, let $q_l\colon E^r_B\to E$ be the projection onto the $l$-th factor. Define 
\[
\omega^l_{ij}:=q_l^*(\omega_{ij})\in H^*(E^r_B;\Z).
\]
A certain collection of some finite products of the classes $\omega^l_{ij}$ gives an additive basis for $H^*(E^r_B;\Z)$ --- see~\cite[Proposition 7.3]{F-P1} for details (and also~\cite{C-F-W} for the case $r=2$). As it turns out, the hypotheses of the Leray--Hirsch theorem (see, for example,~\cite[Theorem 4D.1]{Hatcher}) are satisfied by the fibration $p_r\colon E^r_B\to B$ defined as $p_r(e_1,\ldots,e_r)=p(e_1)$. Therefore, $H^*(E^r_B;\Z)$ has a free basis. In particular, $H^*(E^r_B;\Z)$ is torsion-free and so, its integral cohomology coincides with its rational cohomology.

\begin{lemma}\label{lem: fp zero-divisor}
    For integers $i,j,k,l$ satisfying $1\le i<j\le m+n$ and $1\le l,k\le r$, we have that $\omega_{ij}^l-\omega_{ij}^k \in \textup{Ker}(\Delta_r^*\colon H^*(E^r_B;\Q)\to H^*(E;\Q))$, where $\Delta_r\colon E\to E^r_B$ is the diagonal map.
\end{lemma}

The above statement appears in~\cite{F-P1} as an extension to the statement for the case $r=2$ from~\cite[Proposition 9.4]{C-F-W} in integral cohomology. Since all rings involved are torsion-free, the same is true in rational cohomology as well.

\begin{proof}[Proof of~\Cref{prop: f-n fibration}]
    First, suppose $d$ is odd. Consider the rational cohomology classes
    \[
    x_1=\prod_{i=2}^m\left(\omega^1_{i(m+1)}-\omega^2_{i(m+1)}\right), \hspace{3mm} x_2=\prod_{j=m+1}^{m+n}\left(\omega^2_{1j}-\omega^1_{1j}\right)^2, \hspace{3mm} x_3=\prod_{l=3}^r \prod_{j=m+1}^{m+n}\left(\omega^l_{1j}-\omega^1_{1j}\right).
    \]
As explained in~\cite[Proposition 8.3]{F-P1}, we have $x_1x_2x_3\ne 0$ in $H^*(E^r_B;\Q)$. The non-zero cup product $x_1x_2x_3$ consists of $rn+m-1$ positive-degree classes sitting inside $\text{Ker}(\Delta_r^*)$ in view of~\Cref{lem: fp zero-divisor}. Hence,~\Cref{prop: zcl bound} implies that
\[
rn+m-1\le c\ell_{\Q}(\text{Ker}(\Delta_r^*))\le \dTC_r[p\colon E\to B].
\]
On the other hand, $\TC_r[p\colon E\to B]=rn+m-1$ holds due to~\cite[Theorem 8.1]{F-P1}. This proves the equality for $d$ odd.

We proceed similarly in the case when $d$ is even. We consider the following classes in $\text{Ker}(\Delta_r^*)\subset H^*(E^r_B;\Q)$:
    \[
    y_1=\prod_{i=2}^m\left(\omega^1_{i(m+1)}-\omega^2_{i(m+1)}\right), \hspace{3mm} y_2=\prod_{j=m+2}^{m+n}\left(\omega^1_{(j-1)j}-\omega^2_{(j-1)j}\right), \hspace{3mm} y_3=\prod_{l=2}^r \prod_{j=m+1}^{m+n}\left(\omega^l_{1j}-\omega^1_{1j}\right).
    \]
Due to~\cite[Proposition 9.1]{F-P1}, we have $y_1y_2y_3\ne 0$ in $H^*(E^r_B;\Q)$. This non-zero product has $rn+m-2$ positive-degree classes. Hence,~\Cref{prop: zcl bound} gives the lower bound
\[
rn+m-2\le c\ell_{\Q}(\text{Ker}(\Delta_r^*))\le \dTC_r[p\colon E\to B].
\]
On the other hand, $\TC_r[p\colon E\to B]=rn+m-2$ holds due to~\cite[Theorem 1.3]{F-P2}. This proves the equality for $d$ even.
\end{proof}

\begin{remark}\label{rem: different from fiber}
  \rm{For a fibration $p\colon E\to B$, the gap between the $r$-th sequential parametrized distributional complexity $\dTC_r[p\colon E\to B]$ and the $r$-th sequential (non-parametrized) distributional complexity $\dTC_r(F)$ of its fiber $F$ can be arbitrarily large! Indeed, we consider the Fadell--Neuwirth fibration $p\colon F(\R^d,m+n)\to F(\R^d,m)$ whose fiber is $F(\R^d\setminus \mathcal{O}_m,n)$. It follows from~\cite{Gonzalez-Grant} that
    \[
    \dTC_r(F(\R^d\setminus \mathcal{O}_m,n))\le\TC_r(F(\R^d\setminus \mathcal{O}_m,n))\le rn,
    \]
    see also~\cite{F-G-Y} for the case $r=2$. Hence,~\Cref{prop: f-n fibration} implies that 
    \[
    \dTC_r[p:F(\R^d,m+n)\to F(\R^d,m)]- \dTC_r(F(\R^d\setminus \mathcal{O}_m,n))\ge m-2
    \]
    for all choices of $r,m,d\ge 2$ and $n\ge 1$. We also remark that actually, there is the equality $\dTC_r(F(\R^d\setminus \mathcal{O}_m,n))=\TC_r(F(\R^d\setminus \mathcal{O}_m,n))$. This can be seen as follows.

    The inclusion $F(\R^d\setminus \mathcal{O}_m,n)\hookrightarrow F(\R^d,m+n)$ induces an epimorphism in integral cohomology, see~\cite[Theorem~V.4.2]{F-H}. Hence, $H^*(F(\R^d\setminus \mathcal{O}_m,n);\Z)$ is torsion-free, and so, the integral and rational cohomologies of $F(\R^d\setminus \mathcal{O}_m,n)$ coincide. It has been shown in~\cite{Gonzalez-Grant},~\cite{F-G-Y} that the $r$-th \emph{integral} zero-divisor cup-length\hspace{0.3mm}\footnote{\hspace{0.3mm}We recall that for a space $X$, ring $R$, and integer $r\ge 2$, the $r$-th $R$-zero-divisor cup-length of $X$ is the quantity $c\ell_R(\text{Ker}(\Delta_r^*))$, where $\Delta_r\colon X\to X^r$ is the diagonal map, see~\cite{F},~\cite{Rudyak},~\cite{B-G-R-T}.} of $F(\R^d\setminus \mathcal{O}_m,n)$ is equal to $\TC_r(F(\R^d\setminus \mathcal{O}_m,n))$. Following the techniques of~\cite[Section 7]{Jau1} and using~\Cref{prop: dj epi result}, one can show that $\dTC_r(F(\R^d\setminus \mathcal{O}_m,n))$ is equal to the $r$-th \emph{rational} zero-divisor cup-length of $F(\R^d\setminus \mathcal{O}_m,n)$, and hence equal to $\TC_r(F(\R^d\setminus \mathcal{O}_m,n))$.} 
\end{remark}

\subsection{Calculations for sphere bundles over finite CW complexes}\label{subsec: sphere bundle}
Let $\xi\colon E\to B$ be a locally trivial vector bundle of rank $q\ge 2$ such that there exists a continuous, positive definite scalar product on its fibers. Let $\dot E$ be the set of vectors in $E$ of length $1$. We denote by $\dot\xi\colon\dot E(\xi)\to B$ the associated sphere bundle whose fibers are the unit spheres $S^{q-1}$. Let $\ddot E(\xi)$ be the total space of the bundle of Stiefel manifolds associated with $\xi$, i.e., $\ddot E(\xi)$ consists of pairs $(e,e')$, where $e,e'\in \dot E(\xi)$ are unit vectors such that $e\perp e'$ and $\dot\xi(e)=\dot\xi(e')$. The projection $\ddot \xi\colon \ddot E(\xi)\to \dot E(\xi)$ onto the first $\dot E(\xi)$ factor is a bundle with fiber $S^{q-2}$. 

In what follows, we will use the notations $\dot E$ and $\ddot E$ when the bundle $\xi$ is clear.

\begin{example}\label{ex: complex}
    \rm{We always have $r-1\le\dTC_r(S^{q-1})\le\dTC_r[\dot\xi\colon\dot E\to B]\le\TC_r[\dot\xi\colon\dot E\to B]$ by~\Cref{cor: rationally acyclic}. If $\xi\colon E\to B$ is a complex vector bundle, then this lower bound is attained and we have the minimum value $\dTC_r[\dot\xi\colon\dot E\to B]=r-1$. This holds because $\TC_r[\dot\xi\colon\dot E\to B]=r-1$ in this case, see~\cite[Corollary 5.1]{F-P-sphere-bundle}.}
\end{example}

If $\xi\colon E\to B$ is oriented, then the \emph{Euler class} $\mathfrak{e}(\xi)\in H^q(B;\Z)$ is defined as the first obstruction to a section of $\xi$. Similarly, the Euler class $\mathfrak{e}(\ddot\xi)\in H^{q-1}(\dot E;\Z)$ is defined for the oriented projection $\ddot\xi\colon\ddot E\to\dot E$. The maximum integer $k$ such that $\mathfrak{e}(\ddot\xi)^k\ne 0$ is called the \emph{height} of $\mathfrak{e}(\ddot\xi)$, denoted $\mathfrak{h}(\mathfrak{e}(\ddot\xi))$.

We focus on an oriented vector bundle $\xi\colon E\to B$ for which the corresponding sphere bundle $\dot\xi\colon \dot E\to B$ admits a section. In this case, $\mathfrak{e}(\xi)=0$. Hence, the oriented sphere bundle $\dot\xi$ admits a \emph{cohomological extension} of the fiber, i.e., there exists a cohomology class $\mathfrak{u} \in H^{q-1}(\dot E;\Z)$ such that the restriction $\iota_F^*(\mathfrak{u})$ forms a basis for $H^{q-1}(F;\Z)$ in each fiber $F\simeq S^{q-1}$, where $\iota_F\colon F\hookrightarrow \dot E$ is the fiber inclusion, see~\cite[Section 5.7]{Sp}. In particular, any such oriented sphere bundle $\dot\xi$ that admits a section satisfies the conclusion of the Leray--Hirsch theorem,~\cite[Theorem 4D.1]{Hatcher}.

In this subsection, we prove the following.

\begin{proposition}\label{prop: sphere bundle}
    Let $B$ be a finite CW complex and $\xi\colon E\to B$ be a locally trivial oriented vector bundle of rank $q\ge 2$ whose corresponding oriented sphere bundle $\dot\xi\colon\dot E\to B$ admits a section. If $H^*(B;\Z)$ is torsion-free, then
    \[
    \dTC_r[\dot\xi\colon\dot E\to B]\ge \mathfrak{h}(\mathfrak{e}(\ddot\xi))+r-1.
    \]
    The inequality is saturated when
    \begin{enumerate}
        \item $\dim(B)\le (q-1)\cdot\mathfrak{h}(\mathfrak{e}(\ddot\xi))$, or
        \item $\dim(B)<q-1$ and $q$ is odd.
    \end{enumerate}
\end{proposition}

The same lower bounds to the (sequential) parametrized topological complexities of general oriented sphere bundles that admit sections were obtained in~\cite{F-W} and~\cite{F-P-sphere-bundle}.

As a consequence of this bound, we see that the sequential parametrized distributional complexities of oriented sphere bundles can be arbitrarily large (\emph{cf.}~\Cref{rem: arbitrarily large}).

We now proceed to prepare for the proof of~\Cref{prop: sphere bundle}. Let us fix an oriented vector bundle $\xi\colon E\to B$ of rank $q\ge 2$, where $B$ is a finite CW complex such that the corresponding oriented $S^{q-1}$-bundle $\dot\xi\colon\dot E\to B$ admits a section $s\colon B\to \dot E$. Suppose $\eta\colon \dot E(\eta)\to B$ is the bundle of unit vectors in $\dot E$ orthogonal to $s$. More precisely, $\dot E(\eta)=\{e\in \dot E\mid e\perp s(\xi(e))\}$ and $\eta$ is the restriction of $\dot\xi$ on $\dot E(\eta)$. The fibers of $\eta$ are the unit spheres $S^{q-2}$. We note that the Euler class $\mathfrak{e}(\eta)\in H^{q-1}(B;\Z)$ exists. In fact, a cohomological extension $\mathfrak{u}\in H^{q-1}(\dot E;\Z)$ of $\dot\xi$ can be chosen such that
\[
\mathfrak{e}(\eta)=s^*(\mathfrak{u})\in H^{q-1}(B;\Z),
\]
see~\cite[Theorem 4.1]{F-W}.

For an integer $k\ge 1$, let $\dot E^k_B$ denote the subspace $\dot E\times_B\cdots\times_B \dot E\subset (\dot E)^k$, so that $\dot E=\dot E^1_B$. Fix $r\ge 2$ and for each $1 \le i< j \le r$, define the projection $\rho_{ij}\colon \dot E^j_B\to \dot E^i_B$ that forgets the last $j-i$ coordinates, and its section $\sigma_{ij}\colon \dot E^i_B\to \dot E^j_B$, which is given by
\[
\sigma_{ij}(x_1,\ldots,x_i)=(x_1,\ldots,x_i,x_i,\ldots,x_i),
\]
where the $i$-th coordinate repeats itself $j-i$ times.
Note that $\rho_{i(i+1)}$ is a locally trivial oriented bundle with fiber $S^{q-1}$. We have the following maps:
\begin{equation}\label{eq: big diagram}
\begin{tikzcd}[column sep = large, every arrow/.append style={shift left}]
B \arrow{r}{s}
&
\dot E^1_B \arrow{r}{\sigma_{12}} \arrow{l}{{\dot\xi \vphantom{1}}}
& 
\dot E^2_B \arrow{r}{\sigma_{23}} \arrow{l}{{\rho_{12} \vphantom{1}}}
&
\cdots\cdots \arrow{r}{\sigma_{(r-2)(r-1)}} \arrow{l}{{\rho_{23} \vphantom{1}}}
&
\dot E^{r-1}_B \arrow{r}{\sigma_{(r-1)r}} \arrow{l}{{\rho_{(r-2)(r-1)} \vphantom{1}}}
&
\dot E^r_B \arrow{l}{{\rho_{(r-1)r} \vphantom{1}}}
\end{tikzcd}
\end{equation}
Clearly, the composition $\sigma_{(r-1)r}\circ\cdots\circ\sigma_{12}\colon \dot E\to \dot E^r_B$ is the diagonal map $\dot\Delta_r\colon \dot E\to \dot E^r_B$, and $\rho_{12}\circ\cdots\circ\rho_{(r-1)r}\colon \dot E^r_B\to \dot E$ is the projection onto the first factor. 

For each $1\le i < r$, let $\eta_i\colon \dot E(\eta_i)\to \dot E^i_B$ be the bundle of vectors in $\dot E^{i+1}_B$ orthogonal to $\sigma_{i(i+1)}\colon \dot E^{i}_B\to E^{i+1}_B$. The fibers of $\eta_i$ are the unit spheres $S^{q-2}$. Since $\sigma_{i(i+1)}$ is a section of $\rho_{i(i+1)}\colon E^{i+1}_B\to E^{i}_B$, we can choose a cohomological extension $\mathfrak{u}_i\in H^{q-1}(\dot E^{i+1}_B;\Z)$ of the fiber of $\rho_{i(i+1)}$ determined by $\sigma_{i(i+1)}$ such that the Euler class $\mathfrak{e}(\eta_i)\in H^{q-1}(\dot E_B^i;\Z)$ satisfies
\[
\mathfrak{e}(\eta_i)=\sigma^*_{i(i+1)}(\mathfrak{u}_i)\in H^{q-1}(\dot E_B^i;\Z).
\]
Then in the above notations, we have the following description of the cohomology ring $H^*(\dot E^i_B;\Z)$ (see~\cite[Theorem 6.1]{F-P-sphere-bundle} for a proof).

\begin{theorem}\label{th: fp sphere bundle}
There exists a cohomological extension $\mathfrak{u}\in H^{q-1}(\dot E;\Z)$ of the fiber of $\dot\xi\colon \dot E\to B$ such that the cohomology ring $H^*(\dot E;\Z)$ is the quotient of the ring $H^*(B;\Z)[\mathfrak{u}]$ with respect to the principal ideal $\langle\mathfrak{u}^2-\mathfrak{e}(\eta)\mathfrak{u}\rangle$. Moreover, for each $1\le i<r$, there exists a cohomological extension $\mathfrak{u}_i\in H^{q-1}(\dot E^{i+1}_B;\Z)$ of the fiber of $\rho_{i(i+1)}\colon \dot E^{i+1}_B\to \dot E^i_B$ such that the cohomology ring $H^*(\dot E^{i+1}_B;\Z)$ is the quotient of the ring $H^*(\dot E^i_B;\Z)[\mathfrak{u}_i]$ with respect to the principal ideal $\langle\mathfrak{u}_i^2-\mathfrak{e}(\eta_i)\mathfrak{u}_i\rangle$. 
\end{theorem}

Note the classes $\mathfrak{u}$ and $\mathfrak{u}_i$ exist in view of the sections $s\colon B\to \dot E$ and $\sigma_{i(i+1)}\colon \dot E_B^i\to\dot E_B^{i+1}$, respectively. The proof of~\Cref{th: fp sphere bundle} uses the fact that due to these sections, the Leray--Hirsch theorem gives us $r$ number of graded group isomorphisms 
\[
H^*(\dot E;\Z)=H^*(B;\Z)\otimes H^*(S^{q-1};\Z)\ \text{ and } \ H^*(E^{i+1}_B;\Z)=H^*(E^{i}_B;\Z)\otimes H^*(S^{q-1};\Z),
\]
where $1\le i<r$. In particular, $H^*(\dot E;\Z)$ is torsion-free if and only if $H^*(B;\Z)$ is torsion-free. Therefore, if we assume that $H^*(B;\Z)$ is torsion-free, then $H^*(E^i_B;\Z)$ is torsion-free for all $1\le i \le r$, and hence, the integral and rational cohomologies of $E^r_B$ coincide. 

We can now use the discussion so far to prove ~\Cref{prop: sphere bundle}.

\begin{proof}[Proof of~\Cref{prop: sphere bundle}]
First, note that since $B$ is a finite CW complex, the total space $\dot E$ of the corresponding sphere bundle $\dot\xi\colon\dot E\to B$ is also a finite CW complex. Consequently, $\dot E^r_B$ is a finite CW complex, see~\Cref{rmk: cw structure on erb}. For $1\le i <r$, define $h_i\colon \dot E^r_B\to \dot E$ as the projection onto the $i$-th factor. Let $\eta_i'=h^*_i(\ddot\xi)$ be the bundle over $\dot E^r_B$ induced from $h_i$ so that $\eta_i=\sigma_{ir}^*(\eta_i')$. Note that $\mathfrak{e}(\eta_i')\in H^{q-1}(\dot E^r_B;\Z)$ is well-defined for each $i$. When $q$ is odd, we have $\mathfrak{e}(\eta_i')=2\rho^*_{(i+1)r}(\mathfrak{u_i})-\rho_{1r}^*(\mathfrak{e}(\ddot\xi))$, whereas when $q$ is even, we have $\mathfrak{e}(\eta_i')=\rho_{1r}^*(\mathfrak{e}(\ddot\xi))$ for all $1\le i<r$ --- see~\cite[Corollaries 7.6 \& 7.7]{F-P-sphere-bundle} for these facts. Here, $\mathfrak{u}_i\in H^{q-1}(\dot E^{i+1}_B;\Z)$ is the class from~\Cref{th: fp sphere bundle}. Using this, it has been shown in~\cite[Section 7]{F-P-sphere-bundle} that
 \[
    \prod_{i=1}^{r-1}\left(\rho^*_{(i+1)r}(\mathfrak{u_i})-\mathfrak{e}(\eta_i')\right)^{b_i+1}\ne 0,
    \]
where $b_i\ge 0$ are integers satisfying $\sum_{i=1}^{r-1}b_i=\mathfrak{h}(\mathfrak{e}(\ddot\xi))$. Using the commutativity in~\eqref{eq: big diagram} and the fact that $\mathfrak{e}(\eta_i)=\sigma_{ir}^*(\mathfrak{e}(\eta_i'))$, we can conclude that
\[
\Delta_r^*(\mathfrak{e}(\eta_i'))=\mathfrak{e}(\ddot\xi)=(\rho_{(i+1)r}\circ\Delta_r)^*(\mathfrak{u_i}))
\]
 for all $1\le i<r$. Hence, $\rho^*_{(i+1)r}(\mathfrak{u_i})-\mathfrak{e}(\eta_i')\in\text{Ker}(\dot\Delta_r^*\colon H^*(\dot E^r_B;\Q)\to H^*(\dot E;\Q))$. We note that these are rational cohomology classes because $H^*(B;\Z)$ is torsion-free. Therefore, we get from this and~\Cref{prop: zcl bound} the desired inequality
\begin{equation}\label{eq: sphere bundle bound}
\dTC_r[\dot\xi\colon \dot E\to B]\ge c\ell_{\Q}(\text{Ker}(\dot\Delta_r^*))\ge \mathfrak{h}(\mathfrak{e}(\ddot\xi))+r-1. 
\end{equation}
Next, if $\dim(B)\le (q-1)\cdot\mathfrak{h}(\mathfrak{e}(\ddot\xi))$, then $\TC_r[\dot\xi\colon \dot E\to B]=\mathfrak{h}(\mathfrak{e}(\ddot\xi))+r-1$ due to~\cite[Corollary 8.1]{F-P-sphere-bundle}, hence giving us the same value for $\dTC_r[\dot\xi\colon \dot E\to B]$ because of~\eqref{eq: sphere bundle bound}. Finally, if $q$ is odd and $\dim(B)<q-1$, then $\dim(\dot E)<2(q-1)$, so that $\mathfrak{h}(\mathfrak{e}(\ddot\xi))=1$. This gives the lower bound $r$ from~\eqref{eq: sphere bundle bound}. The upper bound $\TC_r[\dot\xi\colon \dot E\to B]=r$ follows immediately from~\cite[Corollary 8.2]{F-P-sphere-bundle}.
\end{proof}
In the case when the vector bundle $\xi\colon E\to B$ itself admits a section, we have the following simpler lower bound, which uses only the information of $H^*(B;\Q)$.
\begin{corollary}\label{cor: sphere bundle 2}
    Let $\eta\colon E(\eta)\to B$ be an oriented vector bundle of positive rank $q-1$, where $B$ is a finite CW complex. Suppose $\xi=\eta\oplus\varepsilon$ is the Whitney sum, where $\varepsilon$ is the trivial line bundle over $B$. If $H^*(B;\Z)$ is torsion-free, then
    \[
\dTC_r[\dot\xi\colon\dot E(\xi)\to B]\ge \mathfrak{h}(\mathfrak{e}(\eta))+r-1.
\]
Additionally, if $q\ge 3$ is odd and $\mathfrak{h}(\mathfrak{e}(\eta))$ is even, then 
\[
\dTC_r[\dot\xi\colon\dot E(\xi)\to B]\ge \mathfrak{h}(\mathfrak{e}(\eta))+r.
\]
\end{corollary}
\begin{proof}
    Let $s\colon B\to \dot E(\xi)$ be the section of $\dot\xi$ implied by the summand $\varepsilon$. In $H^{q-1}(\dot E (\xi);\Q)$, let $\mathfrak{e}(\ddot\xi)=(\dot\xi)^*(a)+b\mathfrak{c}$ be the unique representation implied by the Leray--Hirsch theorem. Here, $a\in H^{q-1}(B;\Q)$, $b\in\Q$, and $\mathfrak{c}\in H^{q-1}(\dot E(\xi);\Q)$ is a cohomological extension of the fiber of $\dot\xi$ such that $s^* (\mathfrak{c})=\mathfrak{e}(\eta)\in H^*(B;\Q)$. In this situation, it follows from the work done in~\cite{F-W},~\cite{F-W2}, and~\cite{F-P-sphere-bundle} that 
    \[
    \mathfrak{h}(\mathfrak{e}(\ddot\xi))=\begin{cases}
        \mathfrak{h}(\mathfrak{e}(\eta))+1 & \text{ if } \mathfrak{h}(\mathfrak{e}(\eta)) \text{ is even and }q\text{ is odd}
        \\
        \mathfrak{h}(\mathfrak{e}(\eta)) & \text{ otherwise}.
    \end{cases}
    \]
    Note that the sphere bundle $\dot\xi\colon\dot E\to B$ satisfies the conditions of~\Cref{prop: sphere bundle}. Hence, our conclusion follows from~\Cref{prop: sphere bundle}.
\end{proof}

\begin{remark}\label{rem: arbitrarily large}
   \rm{The quantity $\mathfrak{h}(\mathfrak{e}(\ddot\xi))$ can be arbitrarily large, as we explain below. So, the gap between $\dTC_r[\dot\xi\colon\dot E\to B]$ and $\dTC_r(S^{q-1})$ can be as big as desired. In particular, the cohomological lower bound to $\dTC_r[\dot\xi\colon\dot E\to B]$ implied by~\Cref{prop: sphere bundle} is much better than the homotopy-theoretic lower bound $\dTC_r(S^{q-1})\ge r-1$ implied by~\Cref{cor: rationally acyclic} (or, more generally, by~\Cref{prop: fiber lower bound}) in several situations.
    
To see an instance of the large values of $\mathfrak{h}(\mathfrak{e}(\ddot\xi))$, consider the canonical complex line bundle $\eta\colon E(\eta)\to \C P^n$, for which we obviously have $\mathfrak{h}(\mathfrak{e}(\eta))=n$. We can view $\eta$ as a real vector bundle of rank $q-1=2$, so that the rank of $\xi=\eta\oplus\varepsilon$ is $q=3$. Hence,~\Cref{cor: sphere bundle 2} implies $\dTC_r[\dot\xi\colon\dot E(\xi)\to \C P^n]\ge n+r-1$ for each $n\ge 1$ since we have $\mathfrak{h}(\mathfrak{e}(\ddot\xi))\ge n$.

For this canonical line bundle,~\Cref{ex: complex} gives us $\dTC_r[\dot\eta\colon\dot E(\eta)\to \C P^n]=r-1$. So, this example also shows that for a Whitney sum $\xi=\eta\oplus\varepsilon$, the difference between $\dTC_r[\dot\xi\colon\dot E(\xi)\to B]$ and $\dTC_r[\dot\eta\colon\dot E(\eta)\to B]$ can be arbitrarily large. }
\end{remark}

\section{Sequential equivariant distributional complexity}\label{sec: equi dTC}
We now introduce the equivariant analogue of the notion of distributional sectional category (\emph{cf.}~\Cref{sect: dsecat}), which is also the distributional analogue of the classical notion of equivariant sectional category (\emph{cf.}~\Cref{subsec: equi sect cat}), and study some of its properties. 

Suppose $E$ and $B$ are $G$-spaces, and $p\colon E\to B$ is a $G$-fibration. The measure space $\mathcal{B}_n(E^I)$ has a natural $G$-action defined as follows:
$$
g\cdot\sum\lambda_{\phi}\phi=\sum \lambda_{\phi}\hspace{0.5mm}g\phi \ \text{ for } \ g\in G \ \text{ and } \ \phi\in E^I,
$$
where $g\phi\colon I \to E$ is defined as $g\phi(t)=g\cdot\phi(t)$ for each $t\in I$.
This action is well defined, since there is a bijection $\mathrm{supp}(g\cdot \mu)\cong g\hspace{1mm} \mathrm{supp}(\mu)$ for each $\mu\in \B_n(E^I)$ and $g\in G$.
In particular, the space $E_n(p)=\bigcup_{x \in B}\mathcal{B}_n  (p^{-1}(x))$ also admits a $G$-action.

Using similar arguments as in~\cite[Proposition 5.1]{D-J}, one can show that the fibration $\mathcal{B}_n(p)\colon E_n(p)\to B$ is a $G$-fibration using the fact that the $G$-fibration $p$ allows the lift of $G$-equivariant homotopies to $G$-equivariant homotopies.

\begin{definition}
\rm{Let $p\colon E\to B$ be a $G$-fibration. 
The \emph{equivariant distributional sectional category} of $p$, denoted $\dsct_G(p)$, is defined as the smallest integer $n$ such that the $G$-fibration $\mathcal{B}_{n+1}(p) \colon E_{n+1}(p) \to B$ admits a $G$-section.  } 
\end{definition}

The following observations are immediate from the definition:
\begin{enumerate}
    \item for any $G$-fibration $p\colon E\to B$, we have $\dsct(p)\leq \dsct_G(p)\le\secat_G(p)$. The latter is true because the map $\ast^k_BE\to E_k(p)$ mentioned in~\Cref{rmk: dsecat vs secat} is $G$-equivariant for each $k\ge 1$.
    \item if $G$-acts trivially on $E$ and $B$, then $\dsct_G(p)=\dsct(p)$. So, the notion of equivariant distributional sectional category extends the notion of distributional sectional category from~\Cref{def: dsecat}.
\end{enumerate}

The equivariant distributional sectional category is an equivariant homotopy invariant in the following sense. The proof is similar to that of~\Cref{prop: homo inv of dsecat} in the non-equivariant case found in~\cite[Section 5]{Jau1}, so we omit it here.
\begin{proposition}\label{prop: eq homotopy invariace of dsecat}
If $p\colon E\to B$ and $p'\colon E'\to B'$ are $G$-fibrations in the commutative diagram
\[\begin{tikzcd}
E \arrow[d, "p"']  \arrow[r, "f" shift left=1.5ex] 
 & E' \arrow[d, "p'"]   \\
B \arrow[r, "g" shift left=1.5ex]& B',
\end{tikzcd}\]
where $f$ and $g$ are $G$-homotopy equivalences, then $\dsct_G(p)=\dsct_G(p')$.
\end{proposition}

It is known that the free path space fibration $\pi^X_r$ defined in~\eqref{eq: sequential free path fibration} is a $G$-fibration if $X$ is a $G$-space (see~\cite[Proposition 4.2]{D-EqPTC} for an explicit proof). 

\begin{definition}\label{def: obvious}
    \rm{The $r$\emph{-th sequential equivariant distributional complexity} of a $G$-space $X$, denoted $\dTC_{G,r}(X)$, is defined as $\dTC_{G,r}(X):=\dsct_G(\pi_r^X)$.}
\end{definition}
In particular, the \emph{equivariant distributional complexity} of a $G$-space $X$, which we denote simply by $\dTC_G(X)$, is the smallest non-negative integer $n$ for which there is an equivariant $(n+1)$-distributed navigation algorithm $s\colon X\times X\to X^I_{n+1}(\pi_2^X)\subset \B_{n+1}(X^I)$.

As a result of~\Cref{prop: eq homotopy invariace of dsecat}, we obtain the $G$-homotopy invariance of $\dTC_{G,r}$. 

\begin{proposition}
Suppose $X$ and $Y$ are $G$-homotopy equivalent $G$-spaces. Then $$\dTC_{G,r}(X)=\dTC_{G,r}(Y).$$    
\end{proposition}

We have the following exponential upper bound to $\dTC_{G,r}$ in terms of $\dTC_{G}$.

\begin{proposition}\label{prop:obvious}
    For any $r\ge 2$ and $G$-space $X$, we have $\dTC_{G,r}(X)\le(\dTC_G(X)+1)^{r-1}-1$.
\end{proposition}

The proof of the above statement is very similar to that of the statement in the non-equivariant case, see~\cite[Proposition 3.9]{Jau1}.

\begin{remark}\label{rmk: simple properties}
\rm{Let $X$ be a $G$-space.
\begin{enumerate}
\item The inequality $\dTC_r(X)\leq \dTC_{G,r}(X)\le\TC_{G,r}(X)$ is obvious.

\item Suppose $K$ is a subgroup of $G$. Then we have $\dTC_{K,r}(X)\leq \dTC_{G,r}(X)$.

\item Suppose $H$ is a non-trivial subgroup of $G$, and let $K$ be a subgroup of $G$ such that $X^H$ admits a $K$-action. In this case, the inequality $\TC_{K,r}(X^H) \leq \TC_{G,r}(X)$ holds (see~\cite{EqTC},~\cite{B-S}). However, we do not know whether $\dTC_{K,r}(X^H) \leq \dTC_{G,r}(X)$ is true in general for any $r\ge 2$. This is because while a $G$-equivariant $r$-th sequential $k$-distributed navigation algorithm $s \colon X^r \to \mathcal{B}_k(X^I)$ maps $H$-fixed points to $H$-fixed points and hence induces a map $s'\colon {(X^H)^r} \to (\mathcal{B}_k(X^I))^H$, there need not be a map $(\mathcal{B}_k(X^I))^H \to \mathcal{B}_k((X^H)^I)$ for $k\ge 2$ since $H\ne\{1\}$. So, it is unclear how the required $K$-equivariant section $(X^H)^r\to \B_k((X^H)^I)$ can be obtained from $s$.
\end{enumerate}}
\end{remark}

\subsection{Sequential equivariant distributional complexities of products of spheres}\label{subsec: eq dtc of products of spheres}
In this section, we provide some estimates and exact computations of the sequential $\Z_2$-equivariant (distributional) topological complexities of the product of spheres $F=\prod_{i=1}^mS^{n_i}$ for various $\Z_2$-actions for which $F$ is $\Z_2$-connected. These computations will be used later to compute the sequential parametrized distributional complexities of some fiber bundles whose fibers are $F=\prod_{i=1}^mS^{n_i}$.

Let $H$ be a closed subgroup of $G$, and let $X$ be a $G$-space. Recall that the set of $H$-fixed points of $X$ is denoted by $X^H$ and defined as
\[
X^H := \{x \in X \mid h \cdot x = x \text{ for all } h \in H\}.
\]
We recall that a $G$-space $X$ is said to be \emph{$G$-connected} if for any closed subgroup $H$ of $G$, the $H$-fixed point set $X^H$ is path-connected.

\begin{proposition}\label{prop: eq tc all antipodal}
Suppose $\Z_2$ acts on each $S^{n_i}$ via antipodal involution for $n_i\geq 1$ and diagonally on the product $\prod_{i=1}^mS^{n_i}$.  Then 
$$
\TC_{\Z_2,r}\left(\prod_{i=1}^mS^{n_i}\right)=m(r-1)+\ell_m,
$$
where $\ell_m$ is the number of $n_i$'s that are even.
\end{proposition}
\begin{proof}
Using the product inequality from~\cite[Proposition 2.14]{A-D} and the computation for $\TC_r$ of products of spheres from~\cite[Corollary 3.12]{B-G-R-T}, we have that
$$
m(r-1)+\ell_m\leq\TC_r\left(\prod_{i=1}^mS^{n_i}\right)\leq \TC_{\Z_2,r}\left(\prod_{i=1}^mS^{n_i}\right)\leq \sum_{i=1}^m\TC_{\Z_2,r}(S^{n_i}).
$$
To obtain the reverse inequality, it suffices to show that $\TC_{\Z_2,r}(S^{n_i})=\TC_{r}(S^{n_i})$. This can be proved using the open sets from~\cite[Section 4]{Rudyak} by proceeding along the lines of the proof of~\cite[Lemma 4.1]{G-G-T-X}. 
\end{proof}

We now consider a general $\Z_2$-action on spheres. For each $1 \le i \le m$, consider an involution $\tau_i$ on $S^{n_i} \subset \R^{n_i+1}$ defined as follows:
\begin{equation}\label{eq: general invo on sphere}
  \tau_i(x_1, \ldots, x_{p_i-1}, x_{p_i}, \ldots, x_{n_i+1})= (x_1, \ldots, x_{p_i-1}, -x_{p_i}, \ldots, -x_{n_i+1})
\end{equation}
for some integer $p_i$ between $1$ and $n_i+1$. Note that if $p_i=1$, then $\tau_i$ acts antipodally on $S^{n_i}$ (this is the action we studied in~\Cref{prop: eq tc all antipodal}), and if $p_i=n_i+1$, then $\tau_i$ is a reflection across the hyperplane $x_{n_i+1}=0$ in $\R^{n_i+1}$. In any case, we have a $\Z_2$-action on the product $\prod_{i=1}^m S^{n_i}$ via the product involution $\tau:=\tau_1\times \dots \times \tau_{m}$.

\begin{proposition}\label{prop: eq tc general involution}
In the above notations, suppose $p_i\geq 2$ for all $1\leq i\leq m$. Then 
$$
m(r-1)+\ell_m\leq \TC_{\Z_2,r}\left(\prod_{i=1}^m S^{n_i}\right)\leq rm,
$$
where $\ell_m$ is the number of $n_i$'s which are even.  The upper bound is attained if $\ell_m=m$. 
\end{proposition}
\begin{proof}
The left inequality is obvious (\emph{cf.} the proof of~\Cref{prop: eq tc all antipodal}), so we show the other inequality. Note that $(S^{n_i})^{\Z_2}\cong S^{p_i-1}$. Since $p_i\geq 2$,  $S^{n_i}$ is $\Z_2$-connected. Consequently, $\prod_{i=1}^m S^{n_i}$ is also $\Z_2$-connected. Then from~\cite[Proposition 3.17]{B-S}, we have 
\[
\TC_{\Z_2,r}\left(\prod_{i=1}^m S^{n_i}\right)\leq r\hspace{1mm}\cat_{\Z_2}\left(\prod_{i=1}^m S^{n_i}\right).
\]
Since $S^{n_i}$ is not $\Z_2$-contractible, $\cat_{\Z_2}(S^{n_i})\geq 1$. Moreover, $\cat_{\Z_2}(S^{n_i})\leq 1$ as the sets $U_0=S^{n_i}\setminus \{e_1\}$ and $U_1=S^{n_i}\setminus \{-e_1\}$ form a $\Z_2$-categorical\hspace{0.3mm}\footnote{\hspace{0.3mm}We recall that a $G$-invariant subset $A$ of a $G$-space $X$ is called $G$-\emph{categorical} if the inclusion $A\hookrightarrow X$ is $G$-homotopic to a map with values in a single orbit --- see, for example,~\cite{EqTC}.} open cover for $S^{n_i}$, where $e_1=(1,0,0,\dots,0)$. Then, using the product inequality from~\cite[Proposition 2.14]{A-D} in the context of equivariant LS-category, we obtain $\cat_{\Z_2}(\prod_{i=1}^m S^{n_i})\leq m$.
\end{proof}

\begin{remark}
    \rm{We note that the same estimates as in~\Cref{prop: eq tc all antipodal} and~\Cref{prop: eq tc general involution} are true for $\dTC_{\Z_2,r}$ of product of spheres under the respective actions because of the same cohomological lower bound $m(r-1)+\ell_m$.}
\end{remark}

We now describe a class of fiber bundles for which Farber and Oprea~\cite{F-O} established an upper bound on sequential parametrized topological complexity in terms of the sequential equivariant topological complexity of their fibers.

Suppose $q \colon P\to B$ is a principal $G$-bundle. For any $G$-space $F$, one can construct an associated locally trivial bundle $p\colon (F\times \hspace{0.5mm}P)/G\to P/G$ with fiber $F$ and structure group $G$. For brevity, we denote $(F\times P)/G$ by $F\times_G\hspace{0.5mm}P$, and call $p$ the \emph{$F$-associate} of $q$. While it is true that $\TC_r[q\colon P\to B]=\TC_r(G)$ (see~\cite{F-P1}, and also~\cite{C-F-W} for the case $r=2$), Farber and Oprea  showed in~\cite[Theorem 3.4]{F-O} for the associated bundles as above that

\begin{equation}\label{eq: f-o inequality} 
\TC_r[p\colon F\times_G \hspace{0.3mm}P\to P/G]\leq \TC_{G,r}(F). 
\end{equation}

For such associated bundles, Proposition~\ref{prop: fiber lower bound} then implies
\begin{equation}\label{eq: imp inequalities}
    \dTC_r(F)\le\dTC_r[p\colon F\times_G \hspace{0.3mm}P\to P/G]\leq \TC_r[p\colon F\times_G \hspace{0.3mm}P\to P/G]\leq \TC_{G,r}(F).
\end{equation}
In~\Cref{thm: distributional analogue of Farber-Oprea}, we will obtain a better upper bound to $\dTC_r[p\colon F\times_G \hspace{0.3mm}P\to P/G]$ for the $F$-associates of principal $G$ bundles $q\colon P\to B$. For now, we study the sequential parametrized distributional complexity of these bundles with fiber $F=\prod_{i=1}^mS^{n_i}$ using~\eqref{eq: imp inequalities}.

\begin{corollary}
Suppose $\Z_2$ acts on each $S^{n_i}$ via antipodal involution for $n_i\geq 1$ and diagonally on the product $F=\prod_{i=1}^mS^{n_i}$.  Then 
$$
\dTC_{r}[p\colon F\times_{\Z_2}P\to P/\Z_2]=m(r-1)+\ell_m,
$$
where $\ell_m$ is the number of $n_i$'s which are even and $p$ is the $F$-associate of a principal $\Z_2$-bundle. 
\end{corollary}
\begin{proof}
    This follows from~\eqref{eq: imp inequalities},~\Cref{prop: eq tc all antipodal}, and~\cite[Proposition 7.8]{Jau1}, where the latter says that $\dTC_r(F)=\TC_r(F)=m(r-1)+\ell_m$. 
\end{proof}

\begin{corollary}
Suppose $\Z_2$ acts on $S^{n_i}$ via the involution described in~\eqref{eq: general invo on sphere} and on $F=\prod_{i=1}^mS^{n_i}$ via the product involution $\tau=\tau_1\times \dots \times \tau_{m}$. If $p_i\geq 2$ for all $1\leq i\leq m $, then we have  
\[
m(r-1)+\ell_m\leq \dTC_r[p\colon F\times_{\Z_2}P\to P/\Z_2] \leq rm,
\]
where $\ell_m$ is the number of $n_i$'s which are even, and $p$ is the $F$-associate of a principal $\Z_2$-bundle. The upper bound is attained if $\ell_m=m$.
\end{corollary}
\begin{proof}
This follows from~\eqref{eq: imp inequalities},~\Cref{prop: eq tc general involution}, and~\cite[Proposition 7.8]{Jau1}.
\end{proof}

Of course, the same proofs work for $\TC_r[p\colon F\times_{\Z_2}P\to P/\Z_2]$ as well.

\subsection{Equivariant distributional complexity of real projective spaces}\label{subsec: equi dtc rpn}

In this section, we study an $SO(n)$-action on $\R P^n$ which will help us obtain new examples of fibrations in~\Cref{sec:new upper bound} whose sequential parametrized distributional complexities differ from their corresponding sequential parametrized topological complexities.

Suppose $SO(n)$ embeds in $SO(n+1)$ in the standard way via 
$$
A\mapsto \begin{bmatrix}
A & 0 \\
0 & 1
\end{bmatrix}.
$$
The natural action of $SO(n+1)$ on $\R^{n+1}$ induces an action of $SO(n)$ on $\R^{n+1}$. Note that under this action, the last coordinate of each vector in $\R^{n+1}$ remains unchanged. In other words, $SO(n)$ fixes the last coordinate axis of $\R^{n+1}$. Indeed, the fixed point set is given by
$$(\R^{n+1})^{SO(n)}=\{(0,\dots,0,z)\in \R^{n+1}\mid z\in \R\}.$$
Observe that this action is linear. That is, if $g\in SO(n)$, then $g(v+w)=g(v)+g(w)$ and $g(\lambda v)=\lambda g(v)$ for any $\lambda\in \R$. Therefore, for any subgroup $H$ of $SO(n)$, the corresponding $H$-fixed point set $(\R^{n+1})^H$ is a vector subspace. The $n$-th real projective space $\R P^n$ is the quotient space 
\[
\frac{\R^{n+1}\setminus\{\vec{0}\}}{\{v\sim \lambda v\mid \lambda \in \R\}}.
\]
The action of $SO(n)$ induces an action on $\R P^n$ in the following way:
\begin{equation}\label{eq: SOn action on RPn}
 g\cdot [v]=[gv]\ \text{ for } \ g\in SO(n) \ \text{ and }\ [v]\in \R P^n,  
\end{equation}
where $[v]$ and $[gv]$ denote the equivalence classes of $v\in \R^{n+1}\setminus\{\vec{0}\}$ and $gv\in \R^{n+1}\setminus\{\vec{0}\}$, respectively, under the relation $\sim$. 

For this action, we show the following computation of the $SO(n)$-equivariant distributional complexity of $\R P^n$ for each $n\ge 1$.

\begin{proposition}\label{prop: SOn eq TC of RPn}
For any $n\ge 1$, suppose $SO(n)$ acts on $\R P^n$ as described in~\eqref{eq: SOn action on RPn}. Then 
$$
\dTC_{SO(n)}(\R P^n)=1.
$$
\end{proposition}

In particular, $\dTC_{SO(n)}(\R P^n)$ is minimum, since $1=\dTC(\R P^n)\le\dTC_{SO(n)}(\R P^n)$, see~\Cref{rmk: simple properties}. Moreover, for $n\ge 2$, we have
\[
\dTC_{SO(n)}(\R P^n)=1<n=\cat(\R P^n)\le\TC(\R P^n)\le\TC_{SO(n)}(\R P^n).
\]
So, the notion of the $r$-th sequential equivariant distributional complexity is different from that of the $r$-th sequential equivariant topological complexity of~\cite{EqTC},~\cite{B-S} for each $r\ge 2$. For $r\ge 3$, this happens because the strict inequality $\dTC_{SO(n),r}(\R P^n)<\TC_{SO(n),r}(\R P^n)$ holds for all but finitely many $n$ in view of~\Cref{prop:obvious} and~\Cref{prop: SOn eq TC of RPn}.

Let us now prepare to prove~\Cref{prop: SOn eq TC of RPn}. We begin by observing that the fixed point set of the $SO(n)$-action on $\R P^n$ as described in~\eqref{eq: SOn action on RPn} is $(\R P^n)^{SO(n)}=\{[(0,\ldots, 0, 1)]\}$. More generally, we have the following.

\begin{lemma}\label{lem: SOn fixed set}
    Let $H$ be any subgroup of $SO(n)$. Then 
    $$
    (\R P^n)^H=\frac{(\R^{n+1})^H\setminus \{\vec{0}\}}{\{v\sim \lambda v\mid \lambda \in \R\}}.
    $$
\end{lemma}
\begin{proof}
    The inclusion $\frac{(\R^{n+1})^H\setminus \{\vec{0}\}}{\{v\sim \lambda v\mid \lambda \in \R\}}\subseteq (\R P^n)^H$ is obvious. To show the reverse inclusion, we consider $[w]\in (\R P^n)^H$ for $w\in \R^{n+1}\setminus \{\vec{0}\}$, so that $h\cdot [w]=[w]$ for all $h\in H$, which implies that for each $h\in H$, there exists a $\lambda_h\in\R$ such that $hw=\lambda_hw$. But $h$ can only have real eigenvalues $1$ or $-1$. Therefore, we can only have either $hw=w$ or $hw=-w$. Since, under this action, the last coordinate of $w$ is unchanged, $hw\neq -w$. This gives $hw=w$ for each $h$. Consequently, $w\in (\R^{n+1})^H\setminus \{\vec{0}\}$ and $[w]\in \frac{(\R^{n+1})^H\setminus \{\vec{0}\}}{\{v\sim \lambda v\mid \lambda \in \R\}}$.
\end{proof}

\begin{corollary}\label{cor: SO(n)-connected}
With respect to the $SO(n)$-action on $\R P^n$ described in~\eqref{eq: SOn action on RPn}, $\R P^n$ is $SO(n)$-connected for each $n\ge 1$.   
\end{corollary}
\begin{proof}
For any subgroup $H$ of $SO(n)$,~\Cref{lem: SOn fixed set} shows that the corresponding $H$-fixed point set of the resulting $H$-action on $\R P^n$ is a lower-dimensional real projective space $\R P^d$, where $d=\dim((\R^{n+1})^H)-1$. Since $\R P^d$ is path-connected for each $d\ge 0$, the result follows by the definition of $G$-connectedness recalled in~\Cref{subsec: eq dtc of products of spheres}.
\end{proof}

\begin{proof}[Proof of~\Cref{prop: SOn eq TC of RPn}]
 Suppose $x, y \in \R P^n$. Note that both $x$ and $y$ represent lines in $\R^{n+1}$ passing through $\vec{0}$. Suppose $x\ne y$, and $\alpha$ and $\beta$ are the angles generated by these two lines in the plane spanned by them so that $\alpha+\beta=\pi$. There are rotations $R_{\alpha}$ and $R_{\beta}$ which take line $x$ to line $y$ by rotating $x$ via the angles $\alpha$ and $\beta$, respectively. Note that these two rotations define two paths between $x$ and $y$ in $\R P^n$, which we denote by  $\delta_{\alpha}$ and $\delta_{\beta}$, respectively. Then~\cite[Example 3.13]{D-J} (see also~\cite[Proposition 6.9]{K-W}) gives a $2$-distributed navigation algorithm $s\colon \R P^n\times \R P^n\to \mathcal{B}_2((\R P^n)^I)$, which is defined as
 $$
 s(x,y)=\frac{\beta}{\pi}\delta_{\alpha}+\frac{\alpha}{\pi}\delta_{\beta}
 $$
for $x\ne y$. In the case $x=y$, $s(x,x)$ is defined to be the constant path. We claim that $s$ is $SO(n)$-equivariant for the $SO(n)$-action described in~\eqref{eq: SOn action on RPn}. This is clear when $x=y\in \R P^n$, so we consider $x\ne y$ below. 

To see this, let $g\in SO(n)$. Then $g\cdot x$ and $g\cdot y$ are rotated lines in the rotated plane. Note that the angles generated by $g\cdot x$ and $g\cdot y$ are the same as those generated by the lines $x$ and $y$ since the $SO(n)$-action fixes the dot product. In particular, the angles generated by $g\cdot x$ and $g\cdot y$ are $\alpha$ and $\beta$.
Also observe that $g\delta_{\alpha}$ (resp. $g\delta_{\beta}$) is the unique path in the rotated plane which takes the line $g\cdot x$ to the line $g\cdot y$ by rotating $g\cdot x$ via the angle $\alpha$ (resp. $\beta$). 
Thus, we obtain for each $t\in I$ that
\[ s(g\cdot x,g\cdot y)(t)  =  \frac{\beta}{\pi}\left(g\delta_{\alpha}\right)(t)+\frac{\alpha}{\pi}\left(g\delta_{\beta}\right)(t)
 = g\cdot\left(\frac{\beta}{\pi}\delta_{\alpha}(t)+\frac{\alpha}{\pi}\delta_{\beta}(t)\right)=g\cdot s(x,y)(t),
\]
which implies that $s(g\cdot x,g\cdot y)=gs(x,y)\in(\R P^n)^I$ for all $x,y\in \R P^n$. Therefore, from this and~\Cref{cor: SO(n)-connected}, we conclude that $s\colon \R P^n\times \R P^n\to \mathcal{B}_2((\R P^n)^I)$ is an $SO(n)$-equivariant $2$-distributed navigation algorithm on $\R P^n$. This gives $\dTC_{SO(n)}(\R P^n)\le 1$ for each $n$. The reverse inequality is obvious.
\end{proof}

\section{An upper bound to parametrized distributional complexity}\label{sec:new upper bound}
In this section, we will obtain an upper bound on the sequential parametrized distributional complexities of the associated bundles of principal bundles in terms of the sequential equivariant distributional complexities of their corresponding fibers. This will give us examples of non-principal bundles on which the sequential parametrized distributional complexities are strictly less than their classical counterparts from~\Cref{subsect: parametrized TC}.

In the case of any principal $G$-bundle $q\colon P\to B$, we have by~\Cref{prop: seq dtc of principal G bundles} that $\dTC_r[q\colon E\to B]=\dTC_r(G)$, which, indeed, gave us~\Cref{ex: temporary}. Now, for any $G$-space $F$ and the $F$-associate bundle $p\colon F\times_G\hspace{0.4mm}P\to P/G$ of $q$ with fiber $F$ and structure group $G$, we obtain the following upper bound.

\begin{theorem}\label{thm: distributional analogue of Farber-Oprea}
Let $p\colon E\to B$ be a locally trivial bundle with fiber $F$ and structure group $G$, which coincides with the associated bundle of a principal $G$-bundle $q\colon P\to B$, so that $E=F\times_G\hspace{0.5mm} P$ and $B=P/G$. Then for each $r\ge 2$, we have for the $F$-associate $p\colon F\times_G \hspace{0.3mm}P\to P/G$ that
$$
\dTC_{r}[p\colon F\times_G \hspace{0.3mm}P\to P/G]\leq (\dTC_{G,r}(F) +1)^2-1.
$$
\end{theorem}
\begin{proof}
 Consider the following commutative diagram:
\begin{equation}\label{eq: left}
    \begin{tikzcd}
F^I \arrow[r, hook] \arrow[swap]{d}{\pi_r^F} 
& 
F^I\times P \arrow{r}{} 
& 
F^I\times_G\hspace{0.5mm} P \arrow{d}{\pi_r^F\times_G\hspace{0.5mm} 1} \arrow{r}{\cong} 
& E^I_B \arrow{d}{\Pi_r^E} 
\\
F^r \arrow{rr}{}  
& 
&
F^r\times_G\hspace{0.5mm} P \arrow{r}{\cong} 
& 
E^r_B,
    \end{tikzcd}
\end{equation}
where the horizontal maps in the right are the obvious homeomorphisms and $1$ denotes the identity map on $P$. Recall that here, $E=F\times_G \hspace{0.3mm}P$ and $B=P/G$. By definition, we have $\dTC_r[p\colon E\to B]=\dsct(\Pi_r^E)$. Therefore, from~\Cref{prop: homo inv of dsecat}, the equality $\dTC_r[p\colon E\to B]=\dsct(\pi_r^F\times_G\hspace{0.5mm} 1)$ follows. 

Now, suppose $\dTC_{G,r}(F)=n-1$. Then there exits a $G$-map $s\colon F^r\to F^I_n(\pi_r^F)$ that is a section to the $G$-fibration $\mathcal{B}_n(\pi_r^F)\colon F^I_n(\pi_r^F)\to F^r$, see~\Cref{def: obvious}. This gives a map $s\times_G\hspace{0.5mm} \iota\colon F^r\times_G\hspace{0.5mm} P\to F^I_n(\pi_r^F)\times_G\hspace{0.5mm} P_n(1)$, where $\iota\colon P\hookrightarrow P_n(1)$ is the inclusion. Note that the space $P_n(1)$ is defined corresponding to the identity fibration $1\colon P\to P$. Consider
$$
\mathcal{B}_n(F^I)\times \mathcal{B}_n(P)\to \mathcal{B}_{n^2}(F^I\times P)\to \mathcal{B}_{n^2}(F^I\times_G\hspace{0.5mm} P),
$$
where the first map is a \emph{measure multiplication map} and the second map is induced by the quotient map $F^I\times P\to (F^I\times P)/G=F^I\times_G\hspace{0.5mm} P$. The explicit description of this composition $\mathcal{B}_n(F^I)\times \mathcal{B}_n(P)\to \mathcal{B}_{n^2}(F^I\times_G\hspace{0.5mm} P)$  is given by the following:
$$
\left(\sum a_{\phi}\phi, \sum b_pp\right)\mapsto \sum a_{\phi}b_p\left[\phi, p\right],
$$
where $[-,-]$ is the notation for the equivalence classes in the orbit space $F^I\times_G\hspace{0.5mm} P$. This is well-defined since
\[
\sum_{\phi,p}a_\phi b_p=\left(\sum_{\phi} a_\phi\right)\left(\sum_pb_p\right)=1\cdot1=1.
\]
Moreover, it is easy to see that the above composition is $G$-invariant, so that we get a map $\mathcal{B}_n(F^I)\times_G\hspace{0.5mm} \mathcal{B}_n(P)\to \mathcal{B}_{n^2}(F^I\times_G\hspace{0.5mm} P)$, which restricts to a map 
$$
\psi\colon  F^I_n(\pi_r^F)\times_G\hspace{0.5mm} P_n(1)\to (F^I\times_G\hspace{0.5mm} P)_{n^2}(\pi_r^F\times_G\hspace{0.5mm} 1)
$$
by definition. Using the left part of~\eqref{eq: left}, we obtain the following commutative diagram:
\[ 
    \begin{tikzcd}
F^I_n(\pi_r^F) \arrow{r}{} \arrow[swap]{d}{\mathcal{B}_n(\pi_r^F)} 
& 
F^I_n(\pi_r^F)\times_G\hspace{0.5mm} P_n(1) \arrow{r}{\psi} \arrow{d}{\mathcal{B}_n(\pi_r^F)\times_G\hspace{0.5mm} \mathcal{B}_n(1)}
& 
(F^I\times_G\hspace{0.5mm} P)_{n^2}(\pi_r^F\times_G\hspace{0.5mm} 1) \arrow{d}{\mathcal{B}_{n^2}(\pi_r^F\times_G\hspace{0.5mm} 1)} 
\\
 F^r \arrow{r}{}  
 & 
 F^r\times_G\hspace{0.5mm} P \arrow{r}{\text{Id}} 
 &
 F^r\times_G\hspace{0.5mm} P.
\end{tikzcd}
\]
Since $s\times_G\hspace{0.5mm} \iota$ is a section of $\mathcal{B}_n(\pi^F_r)\times_G\hspace{0.5mm} \mathcal{B}_n(1)$, the composition $\psi\circ (s\times_G\hspace{0.5mm} \iota)$ defines a section of $\mathcal{B}_{n^2}(\pi_r^F\times_G\hspace{0.5mm} 1)$.
This implies $\dsct(\pi_r^F\times_G\hspace{0.5mm} 1)\leq n^2-1$. Hence, we are done.
\end{proof}

\begin{remark}
\rm{While~\Cref{thm: distributional analogue of Farber-Oprea} is a partial distributional analogue to the inequality~\eqref{eq: f-o inequality} from the classical setting, this is all we need here to obtain our examples below.}
\end{remark}

Let $q\colon P\to B$ be an $SO(n)$-bundle. We consider the $SO(n)$-action on $\R P^n$ described in~\eqref{eq: SOn action on RPn} and obtain the $\R P^n$-associate of $q$, namely $p\colon \R P^n\times_{SO(n)} \hspace{0.3mm}P\to P/{SO(n)}$, whose fiber is $\R P^n$ and structure group is $SO(n)$. For these bundles, we have the following.

\begin{proposition}\label{prop: different examples}
For each $r\ge 2$, there exists $n_r\in\mathbb{N}$ such that there is a strict inequality 
\[
\dTC_r[p\colon \R P^n\times_{SO(n)} \hspace{0.3mm}P\to P/{SO(n)}]<\TC_r[p\colon \R P^n\times_{SO(n)} \hspace{0.3mm}P\to P/{SO(n)}]
\]
for the $\R P^n$-associates of principal $SO(n)$-bundles with $n> n_r$.
\end{proposition}
\begin{proof}
    Take any $n\ge 1$ and $r\ge 2$. By~\Cref{thm: distributional analogue of Farber-Oprea},~\Cref{prop:obvious}, and~\Cref{prop: SOn eq TC of RPn}, we have for the $\R P^n$-associate $p$ the uniform upper bound
\[
\dTC_r[p\colon \R P^n\times_{SO(n)} \hspace{0.3mm}P\to P/{SO(n)}]\le \left(\dTC_{SO(n),r}(\R P^n)+1\right)^2-1 \le 2^{2r-2}-1.
\]
On the other hand, well-known inequalities from~\cite{B-G-R-T} and~\cite{F-P1} give
\[
\TC_r[p\colon \R P^n\times_{SO(n)} \hspace{0.3mm}P\to P/{SO(n)}]\ge \TC_r(\R P^n)\ge \cat((\R P^n)^{r-1})=n(r-1),
\]
where the last equality holds because $\cat(\R P^n)=n$. Let
\[
n_r=\frac{2^{2r-2}-1}{r-1}\ge 3.
\]
Then for all $n>n_r$, we have the required strict inequality.
\end{proof}

\begin{remark}\label{rmk: bott}
\rm{The bundles studied in~\Cref{prop: different examples} indeed exist for each $n>3$. To see this explicitly, we recall that the set of principal $SO(n)$-bundles over the sphere $S^{k+1}$ is in one-to-one correspondence with $\pi_k(SO(n))$ due to~\cite[Theorem 18.5]{Steenrod}. It is well-known that
\[
\pi_3(SO(4))\cong \Z\oplus\Z \ \text{ and } \ \pi_3(SO(n))\cong\Z \ \text{ for all } \ n\ge 5,
\]
see, for example,~\cite{tamura} (this computation for $n\ge 5$ follows from a well-known result of E.~Cartan that $\pi_3(G)\cong\Z$ if $G$ is a compact simple Lie group).
%
Therefore, for each $n\ge 4$, there are infinitely many principal $SO(n)$-bundles over $S^4$, each of which yields a non-trivial $\R P^n$-associate as in~\Cref{prop: different examples}.
}
\end{remark}

In view of~\Cref{rmk: bott}, the computation in~\Cref{prop: different examples} justifies that there are several cases in which the upper bound to $\dTC_{r}[p\colon E\to B]$ implied by~\Cref{thm: distributional analogue of Farber-Oprea} is better than the upper bound from~\eqref{eq: imp inequalities}, and that the gap between the $r$-th sequential parametrized distributional complexity and the $r$-th sequential parametrized topological complexity can be arbitrarily large on non-principal bundles.

\section*{Acknowledgment}
Navnath Daundkar gratefully acknowledges the support of the DST–INSPIRE Faculty Fellowship (Faculty Registration No. IFA24-MA218), Department of Science and Technology, Government of India. The authors thank the anonymous referee for valuable suggestions.

\bibliographystyle{plain} 
\bibliography{references}

\end{document}